\documentclass[a4paper,11pt]{article}

\usepackage{authblk}
\usepackage{amsmath,comment,amsfonts}
\usepackage{bbm}
\usepackage{array}
\usepackage{amssymb}
\usepackage{amscd}
\usepackage{multirow}
\usepackage{calrsfs}
\usepackage{url}
\usepackage{color,colortbl}
\usepackage{soul}
\usepackage{xcolor,yfonts, graphicx,csquotes,amsfonts,amsmath,amssymb,bbm,mathrsfs,bbm}
\usepackage[]{algorithm2e}
\usepackage[normalem]{ulem}
\usepackage{csquotes}
\usepackage[english]{babel}
\usepackage[export]{adjustbox}
\usepackage{float}
\everymath{\displaystyle}
\usepackage{geometry}
\usepackage[font=small]{caption}
\usepackage[labelformat=simple]{subcaption}

\def\mt{t\kern-0.035cm\char39\kern-0.03cm}
\def\ml{l\kern-0.035cm\char39\kern-0.03cm}
\def\md{d\kern-0.035cm\char39\kern-0.03cm}
\def\mL{L\kern-0.08cm\char39}

\newtheorem {definition}{Definition}[section]

\newtheorem {theorem}[definition]{Theorem}
\newtheorem {lema}[definition]{Lemma}

\newtheorem {proposition}[definition]{Proposition}

\newtheorem {example}[definition]{Example}
\newtheorem {remark}[definition]{Remark}
\newtheorem {corollary}[definition]{Corollary}

\newenvironment{itemize*}%
  {\vspace*{-\parskip}\begin{itemize}%
    \setlength{\itemsep}{2pt}%
    \setlength{\parskip}{0pt}}%
  {\end{itemize}}

\newcommand{\proof}{\medskip \noindent {\bf Proof. \ \ }}

\newcommand{\collectionGSF}{\mathcal{E}}

\newcommand{\x}{\mathbf{x}}
\newcommand{\y}{\mathbf{y}}

\newcommand{\vectorA}{\mathbf{a}}
\newcommand{\vectorB}{\mathbf{b}}
\newcommand{\vectorC}{\mathbf{c}}

\newcommand{\bi}{\mathbf{i}}
\newcommand{\bj}{\mathbf{j}}
\newcommand{\bin}{\mathbf{1}}

\newcommand\sA{\mathsf{A}}
\newcommand\aA[2][E]{\sA(#2|#1)}
\newcommand\aAp[2][E]{\sA^{\prime}(#2|#1)}
\newcommand\aAi[3][E]{\sA^{#3}(#2|#1)}

\newcommand\cA{{\mathcal A}}

\newcommand\bA{{\mathscr {A}}}

\newcommand\cE{{\mathcal{E}}}

\newcommand\mmu{\,\underline{\smash{\!\mu\!}}\,}
\newcommand\inv[1]{{\langle#1\rangle}}
\newcommand\blfootnote[1]{%
  \begingroup
  \renewcommand\thefootnote{}\footnote{#1}%
  \addtocounter{footnote}{-1}%
  \endgroup
}

\def\gsf{\mu_{\bA}(\x,\alpha)}

\def\gsf#1{\mu_{\mathcal{A}^{#1}}(\x,\alpha)}

\def\gsf#1#2#3{\mu_{\mathcal{A}^{#1}}(#2,#3)}

\def\gsfp#1#2#3{\mu^\prime_{\mathcal{A}^{#1}}(#2,#3)}

\def\ozn#1{[#1]_0}
\def\fir{{\varphi^*}}
\def\fil{{\varphi_*}}
\def\psir{{\psi^*}}
\def\psil{{\psi_*}}

\usepackage{xcolor}
\definecolor{darkgreen}{rgb}{0,1,0.1}

\definecolor{brown}{rgb}{0.82,0.41,0.15}

\definecolor{brown}{rgb}{0.43, 0.21, 0.1}
\newcommand{\jb}[1]{{\color{red}#1}}
\newcommand{\lh}[1]{{\color{blue}#1}}
\newcommand{\js}[1]{{\color{brown}#1}}
\newcommand{\bas}[1]{{\color{orange}#1}}

\newcommand{\qed}{\null\hfill $\Box\;\;$ \medskip}

\usepackage{enumerate}
\newcommand\gout{\bgroup\markoverwith{\textcolor{green}{\rule[0.5ex]{5pt}{1.5pt}}}\ULon}

\SetKwProg{Fn}{Method}{}{end}

\title{TRALALA}

\newcolumntype{o}{>{\centering\arraybackslash}m{1.2cm}}
\newcolumntype{n}{>{\centering\arraybackslash}m{0.5cm}}
\newcolumntype{N}{>{\centering\arraybackslash}m{0.6cm}}
\newcolumntype{M}{>{\centering\arraybackslash}m{0.7cm}}
\newcolumntype{S}{>{\centering\arraybackslash}m{2.5cm}}

\date{ }

\begin{document}

\begin{center}
{\bf {\sc \Large Conditional aggregation-based Choquet integral\\ on discrete space}}
\end{center}

\vskip 12pt

\begin{center}
{\bf Stanislav Basarik, Jana Borzová, \bf Lenka Hal\v{c}inov\'{a}, Jaroslav \v{S}upina}\blfootnote{{\it Mathematics Subject Classification (2010):} Primary 28A12, 28E10
}\\
{\footnotesize{\textit{Institute of Mathematics, P.~J.~\v{S}af\'arik University in Ko\v sice, Jesenn\'a 5, 040 01 Ko\v{s}ice, Slovakia}}}
\end{center}

\begin{abstract}
We derive computational formulas for the~generalized Choquet integral based on the~novel survival function introduced by M.~Boczek et al.~\cite{BoczekHalcinovaHutnikKaluszka2020}. We demonstrate its usefulness on the~Knapsack problem and the~problem of accommodation options. 
Moreover, we describe sufficient and necessary conditions under which novel survival functions based on different parameters coincide. This is closely related to the incomparability of input vectors (alternatives) in decision-making processes.
\end{abstract}

\noindent{\small{\it Keywords: }{Choquet integral; conditional aggregation operator; decision making, survival function}}

\section{Introduction}


M.~Boczek et al.~\cite{BoczekHalcinovaHutnikKaluszka2020}, inspired by consumer’s problems, aggregation operators, and conditional expectation, introduced a new notion of conditional aggregation operators. Conditional aggregation operators cover many existing aggregations, such as the~arithmetic and geometric mean, or plenty of~integrals known in~the~literature~\cite{KlementMesiarPap2010, Shilkret1971, Sugeno1974}. They became the essence of the~generalization of survival function (a notion known from~\cite{DuranteSempi2015}, from~\cite{grabisch2016set} known as the decumulative distribution function or from~\cite{BorzovaHalcinovaHutnik2015a} known as the strict level measure) introduced by M.~Boczek et al.~\cite{BoczekHalcinovaHutnikKaluszka2020}. Just as the survival function is the basis of the definition of the famous Choquet integral, the~generalized survival function enabled the building of a new integral. The generalized Choquet integral $\mathrm{C}_{\cA}(\x,\mu)$ based on the~generalized survival function $\gsf{}{\x}{\alpha}$ has been naturally introduced by
\[
\mathrm{C}_{\cA}(\x,\mu)=\int_{0}^{\infty}\gsf{}{\x}{\alpha}\,\mathrm{d}\alpha.
\]

A~discrete form of the famous Choquet integral is of great importance in decision making theory, regarding a finite set $[n] = \{1,\dots, n\}$ as criteria set, a vector $\x \in [0, +\infty)^{[n]}$ as a~score vector, and a capacity $\mu\colon 2^{[n]} \to [0, +\infty)$ as the weights of particular sets of criteria. 
We provide problems, which complement the~ones described in~\cite{BoczekHalcinovaHutnikKaluszka2020}, and which stress the~potential of the~novel survival function and integral based on it in decision theory. However, the~main aim of the~present paper is to provide various computational formulas for $\gsf{}{\x}{\alpha}$, and consequently for $\mathrm{C}_{\cA}(\x,\mu)$, since the computation of $\mathrm{C}_{\cA}(\x,\mu)$ has not yet been thoroughly investigated. Thus, we believe that our computational algorithms improve the~practical implementation of the~novel concept of aggregation in various problems.

\par
The~novel survival function $\gsf{}{\x}{\alpha}$ is based on an~aggregation of an~input vector $\x$, applying each conditional aggregation operator from the~family~$\bA$ of size~$\kappa$ on~$\x$ and corresponding set~of its indices. To illustrate our results we provide a~sample formula for $\mathrm{C}_{\cA}(\x,\mu)$. Indeed, we show that 
\begin{align}\label{sample_Ch}{
\mathrm{C}_{\cA}(\x,\mu)=\sum_{i=0}^{\kappa-1}\mu(F_{\bi(i)})(\sA_{i+1}-\sA_i),}
\end{align}
where $\sA_0\leq\dots\leq\sA_{\kappa-1}$ is the~arrangement of all the~values of the~family of conditional aggregation operators~$\bA$ applied on~$\x$ and corresponding sets $E_0,\dots, E_{\kappa-1}$, and $\mu(F_0)\leq\dots\leq\mu(F_{\kappa-1})$ are all the~values of the~measure~$\mu$.
The~index~$\bi(i)$ deciding the~value of $\mathrm{C}_{\cA}(\x,\mu)$ on the~interval $[\sA_i,\sA_{i+1})$ is the~minimum of the~set $\{(0),(1),\dots,(i)\}$ of the~indices such that $F_{(0)}=E_0^c,\dots,F_{(i)}=E_i^c$. The~proof of formula~\eqref{sample_Ch} as well as many other similar formulas can be found in Theorem~\ref{vypocetCh}.
\par
The~paper is organized as follows. Section~2 contains necessary terminology on measure and a~family of conditional aggregation operators. Moreover, it describes two problems in decision-making theory that can be modeled by the~generalized survival function. The~first series of formulas for its computation is presented in~Section~3. The~next section reduces the~minimization of the~reals in the~formulas to the~minimization of indices, and shows two graphical approaches to obtain the~generalized survival function. In addition, it contains a~solution to the~first of two mentioned problems in decision-making theory. The~formulas for the~generalized Choquet integral computation are listed in Section~5, including simplified fomulas for special types of measures. The~solution to the~second problem from Section~2 is presented here. The~study of the~conditions such that the~novel survival functions based on different parameters coincide is performed in~Section~6. Finally, the~paper itself contains just proofs of selected results, the~remaining proofs are attached in the~Appendix.




\section{Background and interpretations}\label{sec: motivation}

As we have already mentioned, we shall consider a finite set $[n]:=\{1,2,\dots,n\}$, $n\in\mathbb{N}$, $n\geq1$. 
Let us denote $\ozn{n}:=\{0\}\cup[n]$.
By $2^{[n]}$ we mean the power set of $[n]$.
A set function $\mu\colon \mathcal{S}\to[0,+\infty)$, $\{\emptyset\}\subseteq\mathcal{S}\subseteq2^{[n]}$, such that $\mu(E)\leq\mu(F)$ whenever $E\subseteq F$, with $\mu(\emptyset)=0$, we call \textit{monotone measure} on $\mathcal{S}$.
Moreover, if $[n]\in \mathcal{S}$, we assume $\mu([n])>0$, and if $\mu([n])=1$, the monotone measure $\mu$ is called \textit{capacity} or \textit{normalized monotone measure}.
Further, we put $\max{\emptyset}=0$, $\min{\emptyset}=+\infty$ and $\textstyle\sum_{i\in\emptyset}x_i=0$.

We shall work with nonnegative real-valued vectors, we use the notation $\x=(x_1,\dots,x_n)$, $x_i\in[0,+\infty)$, $i\in[n]$.
The family of all nonnegative real-valued vectors on $[n]$ is the set $[0,+\infty)^{[n]}$.
By $\bin_E$ we shall denote \textit{indicator function} of a set $E\subseteq [0,+\infty)$, i.e.\ $\bin_E(x)=1$, if $x\in E$, and $\bin_E(x)=0$, if $x\notin E$.
Especially, $\bin_\emptyset(x)=0$ for each $x\in E$.
Let us consider a set $\collectionGSF\subseteq2^{[n]}$.
Unless otherwise stated, for application reasons we assume $\{\emptyset,[n]\}\subseteq\collectionGSF$.
We call it the \textit{collection}.
Let us denote the number of sets in $\collectionGSF$ by $\kappa$, i.e.\ $|\collectionGSF|=\kappa$. 
Let $\hat{\cE}=\{E^c:E\in\cE\}$, i.e.\ $\hat{\cE}$ contains the complements of the sets from collection $\collectionGSF$.
The set of all monotone measures on $\hat{\collectionGSF}$ we shall denote by $\mathbf{M}$.

In the following, we present the definition of the conditional aggregation operator. 
The inspiration for this concept can be found in probability theory, specifically in conditional expectation. 
The idea of aggregating data not on the whole, but on a conditional set is expressed in the terms of this definition. 
The conditional aggregation operator generalizes the classical definition of aggregation operator introduced by Calvo et al.\ in~\cite{CalvoKolesarovaKomornikovaMesiar2002} and forms the basic component of the definition of the generalized survival function.

\begin{definition}\label{def: gsf}\rm(cf.~\cite[Definition 4.1.]{BoczekHalcinovaHutnikKaluszka2020})
A~\textit{family of conditional aggregation operators} (FCA for short) is a~family 
$$\bA=\{\aA{\cdot}: E\in\collectionGSF\}\footnote{$\cA$ is a~family of operators parametrized by a~set from~${\cE}$.},$$
such that each $\aA{\cdot}$ is a~map $\aA{\cdot}\colon [0,+\infty)^{[n]}\to[0,+\infty)$ satisfying the following conditions: 
\begin{enumerate}[\rm(i)]
\item $\aA[E]{\x}\le \aA[E]{\y}$ for any  $\x,\y$ such that $x_i\le y_i$ for any $i\in E$, $E\neq\emptyset$;  
\item  $\aA{\bin_{E^c}}=0$, $E\neq\emptyset$. 
\end{enumerate}
If $\mu$ is a monotone measure on $\hat{\cE}=\{E^c: E\in\cE\}$, i.e. $\mu\in\mathbf{M}$, then the~\textit{generalized survival function} with respect to $\bA$ is defined as
\begin{eqnarray}\label{predpis gsf}\gsf{}{\x}{\alpha}:=\min\left\{\mu(E^c): \aA[E]{\x}\leq\alpha,\, E\in\collectionGSF\right\}\end{eqnarray} for any $\alpha\in[0,+\infty)$.
\end{definition}

If $E\neq\emptyset$, then $\aA{\cdot}$ is called the~\textit{conditional aggregation operator w.r.t.\ $E$}. Moreover, we consider that each element of FCA satisfies $\aA[\emptyset]{\cdot}=0$.

\begin{remark}Note that the survival function $\mu(\{\x>\alpha\})=\mu(\{i\in[n]:x_i>\alpha\})$ can be rewritten as \begin{align}\label{vznik}\mu(\{\x>\alpha\})&=\mu([n]\setminus\{\x\leq\alpha\})=\min\big\{\mu (E^c):(\forall i\in E)\,\, x_i\leq \alpha,\, E\in 2^{[n]}\big\}\nonumber\\ &=\min\{\mu(E^c):\max_{i\in E}x_i\leq\alpha,E\in 2^{[n]}\},\end{align} therefore the introduction of generalized survival function consists in the simple idea of replacing $\max_{i\in E}x_i$~in~(\ref{vznik}) with another functional. Clearly, for $\collectionGSF=2^{[n]}$ and $\cA^{\mathrm{max}}$ we get the original strict survival function.\end{remark}

\begin{example}\label{prikladyFCA}\rm
Let $\x\in[0,+\infty)^{[n]}$. Typical examples of the FCA are:\begin{enumerate}[(i)]
\item $\cA^{\mathrm{max}}=\{\aAi[E]{\cdot}{\mathrm{max}}:E\in{\cE}\}$ with $\aAi[E]{\x}{\mathrm{max}}=\max_{i\in E}x_i$ for $E\neq\emptyset$;
\item $\cA^{\mathrm{min}}=\{\aAi[E]{\cdot}{\mathrm{min}}:E\in{\cE}\}$ with $\aAi[E]{\x}{\mathrm{min}}=\min_{i\in E}x_i$ for $E\neq\emptyset$;
\item $\cA^{\mathrm{sum}}=\{\aAi[E]{\cdot}{\mathrm{sum}}:E\in{\cE}\}$ with $\aAi[E]{\x}{\mathrm{sum}}=\sum_{i\in E}x_i$ for $E\neq\emptyset$.\end{enumerate}\end{example}

Note that the FCA does not have to contain elements of only one type.

\begin{example}
Let $\x\in[0,+\infty)^{[3]}$, and ${\cE}=\{\emptyset, \{1\},\{2\},\{3\},\{1,2,3\}\}$. We can consider the following FCA $$\cA=\{\aAi[E]{\cdot}{\mathrm{max}}:E\in\{\{1\},\{2\},\{1,2,3\}\}\cup\{\aAi[E]{\cdot}{\mathrm{min}}:E\in\{\{3\}\}\}\cup \{\aA[\emptyset]{\cdot}\}.$$
\end{example}
For other examples of FCA we recommend~\cite{BoczekHalcinovaHutnikKaluszka2020}.
On several places in this paper, we shall work with the FCA that is \textit{nondecreasing} w.r.t sets, i.e. the map $E\mapsto\aA[E]{\cdot}$ will be nondecreasing. E.g., the families $\cA^{\mathrm{max}}$ and $\cA^{\mathrm{sum}}$  in Example~\ref{prikladyFCA} are nondecreasing w.r.t. sets.
\bigskip

In the following, we present problems that emphasize the need for the generalized survival function and the generalized Choquet integral in real situations. 
We stress that M.~Boczek et al.\ in~\cite{BoczekHalcinovaHutnikKaluszka2020} introduced several problems of this kind. However, we move beyond these examples.
The solution to these problems using our results is delayed to Subsection~\ref{subsec: indices} and Section~\ref{Choquet}.

\paragraph{Knapsack problem}
Let us imagine a person who is preparing for a holiday. The person plans to travel by plane. Therefore while packing the suitcase, he must keep the rules according to which it is allowed to carry less or equal to $1$~liter of liquids in the suitcase. Moreover, liquids must be in containers with a volume of up to $100$\,ml.
The prices of these products (soap, shampoo, cream, etc.) and their possible combinations (packages of products) are more expensive in the destination country than at home.
The person wants to buy such products, or packages of products, while still at home, to minimize the price of the holiday.

Let $[n]$, $n\geq 1$ be a set of liquid products that the person needs. Then $\collectionGSF\subseteq2^{[n]}$ represents all possible combinations of products. Let us consider $\x=(x_1,\dots,x_n)\in[0,+\infty)^{[n]}$, where $x_i$ represents the volume of $i$-th container and let a monotone measure $\mu\in\mathbf{M}$ represent the price of a package of products. 
Note that the monotone measure $\mu$ need not be additive. It is often possible to buy a package of products that is cheaper than the sum of the prices of the individual products.
The task is to choose such a combination $E\in\collectionGSF$ of products that their volume does not exceed the given limit, i.e.\ $\textstyle\sum_{i\in E}x_i\leq1000$ (milliliters), having in mind that we want to minimize the price of those products that will no longer fit in the suitcase and the person will have to buy them during the holiday.
In other words, we are faced to solve the optimization problem
$$\min\left\{\mu(E^c): \sum_{i\in E}x_i\leq1000,\,E\in\collectionGSF\right\}\text{.}$$

\noindent One can observe that the given formula is a special case of the generalized survival function given in~\eqref{predpis gsf} with the conditional aggregation operator being the sum.

Last, but not least let us point out the essence of the collection $\collectionGSF$. There are situations when instead of the whole powerset one is forced to consider a subcollection  $\collectionGSF\subset 2^{[n]}$. E.g. among the products that one is considering can be shampoo and conditioner separately or shampoo containing conditioner. Of course, these products should not be considered together. Thus all sets consisting of these three products are disqualified and it makes sense to take $\collectionGSF\subset 2^{[n]}$ instead of the whole powerset.

The~full solution to the~problem using our results is accomplished in Subsection~\ref{subsec: indices}.

\paragraph{Problem of accommodation options}
Three people, Anthony, Brittany, Charley, are going on three different holidays and each of them is looking for accommodation in their own destination. 
They use the same online search engine, which searches for available accommodation based on three criteria, namely the distance of accommodation from the destination (in km), the price per night (in euros), and reviews (scale from $1$ to $10$). 
Each person has a certain character (in the sense that the criteria have different importance for them). 
Anthony is a person who saves money, Brittany does not want to walk far and Charley is looking for high-quality accommodation. 
We can model these characters using monotone measures as shown in Table~\ref{charaktery} (assume that the given values are chosen by these three people in the search engine). 
Let us label the criteria as $\text{D}$ (distance), $\text{P}$ (price) and $\text{R}$ (reviews).
\begin{table}[H]
    \centering
    \begin{tabular}{c|c|c|c}
                          & distance (D) & price (P) & reviews (R) \\\hline
         Anthony ($\mu$)  & $0.3$ & $0.8$ & $0.1$ \\
         Brittany ($\nu$) & $0.75$ & $0.4$ & $0.2$ \\
         Charley ($\xi$)  & $0.2$ & $0.5$ & $0.7$
    \end{tabular}
    \caption{Characters of people expressed by a monotone measures $\mu$, $\nu$ and $\xi$}
    \label{charaktery}
\end{table}
\noindent For the sake of simplicity, let us suppose that the search engine offered Anthony, Brittany, and Charley two accommodation options, see Table~\ref{options_booking}. Let the values be in the range that is acceptable for them. 
\begin{table}[H]
    \centering 
    \begin{subtable}{0.32\linewidth}
    \centering
    \begin{tabular}{c|N|N|N}
             & $\text{D}$ & $\text{P}$ & $\text{R}$ \\\hline
       opt.\ $1$ & $4$ & $100$ & $7$\\
       opt.\ $2$ & $10$ & $84$ & $8$
    \end{tabular}
    \caption{Anthony}
    \end{subtable}
    \begin{subtable}{0.32\linewidth}
    \centering
    \begin{tabular}{c|N|N|N}
             & $\text{D}$ & $\text{P}$ & $\text{R}$ \\\hline
       opt.\ $1$ & $7$ & $100$ & $2$\\
       opt.\ $2$ & $10$ & $80$ & $10$
    \end{tabular}
    \caption{Brittany}
    \end{subtable}
    \begin{subtable}{0.32\linewidth}
    \centering
    \begin{tabular}{c|N|N|N}
             & $\text{D}$ & $\text{P}$ & $\text{R}$ \\\hline
       opt.\ $1$ & $3$ & $100$ & $8$\\
       opt.\ $2$ & $15$ & $95$ & $10$
    \end{tabular}
    \caption{Charley}
    \end{subtable}
    \caption{Accommodation options}
    \label{options_booking}
\end{table}
\noindent The aim is to determine the accommodation for each person that would fit him the best (based on character). In other words, the aim is to find a method that will select (based on the character of a person) the accommodation we would expect.

Clearly, this is a problem of the decision-making process. 
In the literature, there is a known approach to solving this problem using the theory of nonadditive measures and integrals, where Choquet integral 
is mainly used, see~\cite{Grabisch1996}. The main advantage of this method is that by nonadditive measures one can model interactions between criteria (e.g. usually with a higher review score one can  expect higher price).
In this paper, we point out the advantages of using a generalized version of the Choquet integral 
in (this) decision-making process.

The~full solution to the~problem using our results is presented in Section~\ref{Choquet}.

\section{
Computational formulas}\label{prvy_sposob}

When working with a~family of conditional aggregation operators~$\bA$, an input vector $\x$ is aggregated on sets $E$ from $\cE\subseteq 2^{[n]}$. By this procedure, instead of $n$ input components of $\x$ we obtain $\kappa=|\cE|$ input values of the~family of conditional aggregation operators~$\bA$:
\begin{align}\label{Ei}
0=\sA_0\leq\sA_1\leq\dots\leq\sA_{\kappa-2}\leq\sA_{\kappa-1}<+\infty,
\end{align}
with the convention $\sA_{\kappa}=+\infty$. This ``new'' input enters the computation of the generalized survival function. Now we provide two approaches to compute generalized survival function values, the~first one is presented in~Theorem~\ref{zjednodusenie_def} and the~other one in~Theorem~\ref{gsf2}. We shall fix a~notation.
\par
Let $E_*$ be any bijection $E_*\colon \ozn{\kappa-1}\to{\cE}$ such that 
\begin{align}\label{EiAi}
\sA_i=\aA[E_{i}]{\x}.
\end{align}
The~second bijection is a~map $F_*\colon \ozn{\kappa-1}\to{\hat{\cE}}$ such that denoting $\mu_j=\mu(F_{j})$ we have:
\begin{align}\label{Fi}
0=\mu_0\leq \mu_1\leq\dots\leq\mu_{\kappa-2}\leq\mu_{\kappa-1}< +\infty.
\end{align}
It should be recalled that $\hat{\cE}$ is the collection of complements of $\cE$, i.e. for every $F_j\in\hat{\cE}$ there exists $E_i\in\cE$ such that $F_j=E_i^c$. We denote sets in collection $\cE$ by $E_i$, $i\in\ozn{\kappa-1}$, and their complements in $\hat{\cE}$ are denoted by $F_j$, $j\in\ozn{\kappa-1}$, because of technical details. In the whole paper for ease of writing, we shall use a~shortcut notation 
$$\sA_{\inv{j}}:=\aA[F_j^c]{\x},\,\,\,\, \text{and}\,\,\,\, \mu_{(i)}:=\mu(E_i^c).$$
One can notice that maps $E_*$ and $F_*$ need not be unique (they are unique just in case they are injective on~$\ozn{\kappa-1}$), but this has no influence on presented results.


In the following example, we only demonstrate the introduction of bijection $E_*$ and $F_*$, respectively.
We shall use the input data below also in other examples in this paper. 

\begin{example}\label{graphical_representation}\rm
Let us consider the collection $\cE=\{\emptyset, \{1\},\{2\}, \{3\},\{1,3\},\{1,2,3\}\}$, the family of conditional aggregation operators  $\cA^{\mathrm{sum}}=\{\aAi[E]{\cdot}{\mathrm{sum}}:E\in\cE\}$, the vector $\x=(2,3,1)$, and the monotone measure $\mu$ on $\hat{\cE}$ with corresponding values in the table. 
\begin{table}[H]
\renewcommand*{\arraystretch}{1.2}
\centering
\begin{tabular}{|c|o|o|o|o|o|o|}
\hline
$E$ & $\emptyset$ & $\{1\}$ & $\{2\}$ & $\{3\}$ & $\{1,3\}$ & $\{1,2,3\}$ \\ \hline
$\aAi[E]{\x}{\mathrm{sum}}$ & $0$  & $2$ & $3$ & $1$ & $3$ & $6$\\\hline
$F$ & $\{1,2,3\}$ & $\{2,3\}$ & $\{1,3\}$ & $\{1,2\}$ & $\{2\}$ & $\emptyset$ \\ \hline
$\mu(F)$ & $1$ & $0.7$ & $0.5$ & $0.5$ & $0.5$ & $0$ \\ \hline
\end{tabular}
\end{table}
\noindent Using maps $E_*$ and $F_*$ we obtain the following arrangement: 
\begin{table}[H]
\renewcommand*{\arraystretch}{1.2}
\centering
\begin{tabular}{|n|o|o|o|o|o|o|}
\hline
$i,j$ & 0 & 1 & 2 & 3 & 4 & 5 \\ \hline
$E_i$ & $\emptyset$ & $\{3\}$ & $\{1\}$ & $\{1,3\}$ & $\{2\}$ & $\{1,2,3\}$ \\\hline
$\sA_i$ & $0$ & $1$ & $2$ & $3$ & $3$ & $6$ \\\hline
$F_j$ & $\emptyset$ & $\{2\}$ & $\{1,2\}$ & $\{1,3\}$ & $\{2,3\}$ & $\{1,2,3\}$ \\ \hline
$\mu_j$ & $0$ & $0.5$ & $0.5$ & $0.5$ & $0.7$ & $1$ \\ \hline
 \end{tabular}
\end{table}
\end{example}
\medskip 

\par

The~first approach to computing generalized survival function values is based on an (ascending) arrangement~\eqref{Ei} of conditional aggregation operator values. Let us note that this approach is directly derived from Definition~\ref{def: gsf}. In the computation process using this approach, we look for generalized survival function values that are achieved at intervals $[\sA_i,\sA_{i+1})$ with the convention $\sA_\kappa=+\infty$.

\begin{theorem}\label{zjednodusenie_def}
Let $\cA$ be a FCA, $\mu\in\mathbf{M}$, $\x\in[0,+\infty)^{[n]}$.
Then \begin{enumerate}[{\rm(i)}] 
\item 
for any $i\in\ozn{\kappa-1}$ it holds that $$\gsf{}{\x}{\alpha} = \min_{k\leq i}\mu_{(k)}\,\,\ \textrm{for any}\,\,\alpha\in[\sA_i,\sA_{i+1});$$ 
\item \belowdisplayskip=0pt \abovedisplayskip=0pt\parbox{\linewidth}{
\begin{align}\label{vyjgsf1}
\gsf{}{\x}{\alpha}=\sum_{i=0}^{\kappa-1} \min_{k\leq i}\mu_{(k)}\bin_{[\sA_i,\sA_{i+1})}(\alpha)\,\,\ \textrm{for any}\,\,\alpha\in[0,+\infty).
\end{align}}

\end{enumerate}
\end{theorem}
\proof (i) Let us consider an arbitrary (fixed) $i\in\ozn{\kappa-1}$ such that $\sA_i<\sA_{i+1}$, and let us take $\alpha\in[\sA_i,\sA_{i+1})$.
Since $0=\sA_0\leq\sA_1\leq\dots\leq\sA_i\leq\alpha$, then $$\gsf{}{\x}{\alpha} =\min\{\mu(E^c): \aA[E]{\x}\leq \alpha,\, E\in\collectionGSF\}=\min\{\mu_{(k)}: k\in\{0,\dots,i\}\}=\min_{k\leq i}\mu_{(k)}.$$

\noindent The case (ii) immediately follows from (i).\qed

\begin{remark} 
Let $(x_1,\dots,x_n)\in[0,+\infty)^{[n]}$. Let us denote 
$\x=(x_{\sigma(1)},x_{\sigma(2)},\dots,x_{\sigma(n)})$, with $\sigma\colon [n]\to[n]$ being a permutation such that $x_{\sigma(1)}\leq x_{\sigma(2)}\leq \dots\leq x_{\sigma(n)}$ with the convention $x_{\sigma(0)}=0$. 
Also let $\cA^\mathrm{max}$ be a FCA with the collection $\collectionGSF=\{G_{\sigma{(i+1)}}^c:i\in\ozn{n}\}$ where $G_{\sigma{(i)}}=\{\sigma{(i)},\dots,\sigma{(n)}\}$ for $i\in[n]$ and $G_{\sigma(n+1)}=\emptyset$. 
Then from the previous proposition we get the standard formula of the survival function~\cite{HalcinovaHutnikKiselakSupina2019}: 
$$\gsf{\mathrm{max}}{\x}{\alpha}=\sum_{i=0}^{n-1} \mu{(G_{\sigma(i+1)})}\bin_{[x_{\sigma{(i)}},x_{\sigma{(i+1)}})}(\alpha)\,\,\ \textrm{for any}\,\,\alpha\in[0,+\infty)\text{.}$$
Indeed, $\sA^\mathrm{max}(\x|G_{\sigma(i+1)}^c)=x_{\sigma(i)}$ for any $i\in\ozn{n}$.
Moreover,  if $x_{\sigma(i)}=\sA^\mathrm{max}_{l_i}$, $l_i\in[\kappa-1]_0$, then $\min_{k\leq l_i}\mu_{(k)}=\min_{k\leq i}\mu(G_{\sigma(k)})=\mu(G_{\sigma(i)})$
because of the monotonicity of $\mu$.
\end{remark}




\begin{remark}\label{gsf=0}\rm
Since $[n]\in\cE$, then the definition of generalized survival function guarantees to achieve zero value. 
According to arrangement in~(\ref{Ei}), for any $\alpha\geq\sA_{\kappa-1}\geq \aA[{[n]}]{\x}$ it holds
$$\gsf{}{\x}{\alpha}=\min\{\mu(E^c): \aA[E]{\x}\leq\alpha,\,E\in\collectionGSF\}\leq\mu([n]^c)=\mu(\emptyset)=0\text{.}$$
On the other hand $\gsf{}{\x}{\alpha}\geq0$.
Thus, the generalized survival function achieves zero value on the interval $[\aA[{[n]}]{\x},+\infty)\supseteq[\sA_{\kappa-1},+\infty)$.
\end{remark}

\begin{example}\label{priklad1}\rm
Let us consider the same inputs as in Example~\ref{graphical_representation}. Using  Theorem~\ref{zjednodusenie_def}(i) let us compute the generalized survival function for any $\alpha\in[0,+\infty)$.  
\begin{itemize*}
    \item $\gsf{\mathrm{sum}}{\x}{\alpha}=\mu_{(5)}=0$ for any $\alpha\in[\sA_5,\sA_6)=[6,+\infty)$, which corresponds to Remark~\ref{gsf=0}.
    \item $\gsf{\mathrm{sum}}{\x}{\alpha}=\mu_{(4)}=0.5$ for any $\alpha\in[\sA_4,\sA_5)=[3,6)$.
    \item $[\sA_3,\sA_4)=[3,3)=\emptyset$.
    \item $\gsf{\mathrm{sum}}{\x}{\alpha}=\mu_{(1)}=0.5$ for any $\alpha\in[\sA_2,\sA_3)=[2,3)$.
    \item $\gsf{\mathrm{sum}}{\x}{\alpha}=\mu_{(1)}=0.5$ for any $\alpha\in[\sA_1,\sA_2)=[1,2)$.
    \item $\gsf{\mathrm{sum}}{\x}{\alpha}=\mu_{(0)}=1$ for any $\alpha\in[\sA_0,\sA_1)=[0,1)$.
\end{itemize*}
\noindent So we get
\begin{align*}
    \gsf{\mathrm{sum}}{\x}{\alpha}&=\bin_{[0,1)}(\alpha)+0.5\cdot\bin_{[1,2)}(\alpha)+0.5\cdot\bin_{[2,3)}(\alpha)+0.5\cdot\bin_{[3,6)}(\alpha)\\&=\bin_{[0,1)}(\alpha)+0.5\cdot\bin_{[1,6)}(\alpha)\text{,}
\end{align*}
$\alpha\in[0,+\infty)$.
\end{example}



The~second approach to computing generalized survival function values is based on an (ascending) arrangement~\eqref{Fi} of measure values. Similarly, as in Remark~\ref{gsf=0}, let us point out first where the generalized survival function achieves zero value.

\begin{remark}\label{A1c}
 From arrangement of $\mu$, see~\eqref{Fi}, we have that 
$\sA_{\inv{0}}\in\{\sA_{\inv{j}}:\mu_j=0,\,j\in\ozn{\kappa-1}\}$.
Thus, for any $\alpha\in[\sA_{\inv{0}},+\infty)$ it holds $\gsf{}{\x}{\alpha}=\mu_0=0$.
\end{remark}

Compared to the first approach presented in~Theorem~\ref{zjednodusenie_def}, here in the second approach we look for intervals at which monotone measure values $\mu(F_j)$ (i.e.\ generalized survival function values) are achieved. It is worth noting that the first approach is appropriate to use if the number of aggregation operator values is much smaller than the number of measure values. In the opposite case, it is more effective to use the second approach. 

\begin{theorem}\label{gsf2}
Let $\cA$ be a FCA, $\mu\in\mathbf{M}$, $\x\in[0,+\infty)^{[n]}$. Then 
\begin{enumerate}[{\rm(i)}]
\item for any $j\in\ozn{\kappa-1}$ it holds that $$\gsf{}{\x}{\alpha} = \mu_j\,\,\ \textrm{for any}\,\,\alpha\in\Big[\min_{k\leq j}\sA_{\inv{k}},\min_{k<j}\sA_{\inv{k}}\Big);$$ 
\item \belowdisplayskip=0pt \abovedisplayskip=0pt\parbox{\linewidth}{
\begin{align}\label{vyjgsf2}
\gsf{}{\x}{\alpha}=\sum_{j=0}^{\kappa-1}\mu_j\bin_{\big[\min\limits_{k\leq j}\sA_{\inv{k}},\min\limits_{k<j}\sA_{\inv{k}}\big)}(\alpha)\,\,\ \textrm{for any}\,\,\alpha\in[0,+\infty).
\end{align}}
\end{enumerate}
\end{theorem}

\proof
For $j=0$ we have, that $\Big[\min_{k\leq 0}\sA_{\inv{k}},\min_{k<0}\sA_{\inv{k}}\Big)=[\sA_{\inv{0}},+\infty)$.
Then, from Remark~\ref{A1c} is easy to see, that for any $\alpha\in[\sA_{\inv{0}},+\infty)$ it holds $\gsf{}{\x}{\alpha}=\mu_0$.
Let us take an arbitrary (fixed) $j\in[\kappa-1]$. The case when $\min_{k\leq j}\sA_{\inv{k}}=\min_{k<j}\sA_{\inv{k}}$ is trivial. Let us suppose that
$\min_{k\leq j}\sA_{\inv{k}}\neq\min_{k<j}\sA_{\inv{k}}$. Then from the fact that
$$\min\Big\{\min_{k<j}\sA_{\inv{k}},\sA_{\inv{j}}\Big\}=\min_{k\leq j}\sA_{\inv{k}}<\min_{k<j}\sA_{\inv{k}},$$ 
we have $\min_{k\leq j}\sA_{\inv{k}}=\sA_{\inv{j}}$. Further, $\min_{k<j}\sA_{\inv{k}}\leq\sA_{\inv{l}}$ for each $l\in[j-1]$. Then for any $\alpha$ such that
$$\sA_{\inv{j}}=\min_{k\leq j}\sA_{\inv{k}}\leq\alpha<\min_{k<j}\sA_{\inv{k}}\leq\sA_{\inv{l}}$$
with $l\in[j-1]$ we have
\begin{itemize}
 \item $\sA_{\inv{j}}\in\{\aA[E]{\x}\leq\alpha,\, E\in\collectionGSF\}$, therefore $\mu_j\in\{\mu(E^c):\aA[E]{\x}\leq\alpha,\,E\in\collectionGSF\}$. Thus, $$\min\{\mu(E^c):\aA[E]{\x}\leq\alpha,\,E\in\collectionGSF\}\leq\mu_j.$$
    \item $\sA_{\inv{l}}\notin\{\aA[E]{\x}\leq\alpha,\, E\in\collectionGSF\}$, therefore $\mu_{l}\notin\{\mu(E^c):\aA[E]{\x}\leq\alpha,\,E\in\collectionGSF\}$ and from the ordering \eqref{Fi} we have $\mu(E^c)\geq \mu_j$ for each $E\in\collectionGSF$ such that $\aA[E]{\x}\leq\alpha$, therefore
    $$\min\{\mu(E^c):\aA[E]{\x}\leq\alpha,\,E\in\collectionGSF\}\geq\mu_j.$$
   
\end{itemize}
The formula in~(ii) directly follows from~(i).
\qed

\begin{example}\label{priklad2}\rm Let us consider the same inputs as in Example~\ref{graphical_representation}. According to Remark~\ref{A1c}, $\gsf{\mathrm{sum}}{\x}{\alpha}=0$ for any $\alpha\in[\sA_{\inv{0}},+\infty)=[6,+\infty)$.
\begin{itemize*}
\item $\gsf{\mathrm{sum}}{\x}{\alpha}=\mu_1=0.5.$ for any $\alpha\in\big[\sA_{\inv{1}},\sA_{\inv{0}}\big)=[3,6)$.
\item $\gsf{\mathrm{sum}}{\x}{\alpha}=\mu_2=0.5$ for any $\alpha\in\big[\sA_{\inv{2}},\sA_{\inv{1}}\big)=[1,3)$.
\item $\gsf{\mathrm{sum}}{\x}{\alpha}=\mu_3=0.5$ for any $\alpha\in\big[\sA_{\inv{2}},\sA_{\inv{2}}\big)=\emptyset$.
\item $\gsf{\mathrm{sum}}{\x}{\alpha}=\mu_4=0.7$ for any $\alpha\in\big[\sA_{\inv{2}},\sA_{\inv{2}}\big)=\emptyset$.
\item $\gsf{\mathrm{sum}}{\x}{\alpha}=\mu_5=1$ for any $\alpha\in\big[\sA_{\inv{5}},\sA_{\inv{2}}\big)=[0,1)$.
\end{itemize*}
Therefore, the generalized survival function has the form
\begin{align*}
    \gsf{\mathrm{sum}}{\x}{\alpha}&=0.5\cdot\bin_{[3,6)}(\alpha)+0.5\cdot\bin_{[1,3)}(\alpha)+\bin_{[0,1)}(\alpha)\\&=\bin_{[0,1)}(\alpha)+0.5\cdot\bin_{[1,6)}(\alpha)\text{,}
\end{align*}
$\alpha\in[0,+\infty)$, compare with Example~\ref{priklad1}.
\end{example}

The following expressions of generalized survival functions are w.r.t.\ special measures. Their~detailed proofs are in Appendix.

\begin{corollary}\label{specimiery} Let $\x\in[0,+\infty)^{[n]}$.
\begin{enumerate}[\rm (i)]
\item Let $\cA$ be FCA and $\bar{\mu}$ be the greatest monotone measure, i.e., $\bar{\mu}(F)=0$, if $F=\emptyset$ and $\bar{\mu}(F)=1$, otherwise. Then the generalized survival function w.r.t.\ $\bar{\mu}$ takes the form $$\bar{\mu}_{\cA}({\x},{\alpha})=\bin_{[0,\aA[{[n]}]{\x})}(\alpha).$$
\
\item Let $\cA$ be FCA and $\mmu$ be the weakest monotone measure, i.e., $\mmu(F)=1$, if $F=[n]$ and $\mmu(F)=0$, otherwise. Then the generalized survival function w.r.t.\ $\mmu$ takes the form $$\mmu_\cA({\x},{\alpha})=\bin_{\big[0,\min\limits_{E\neq\emptyset}\aA[E]{\x}\big)}(\alpha).$$

\item Let $\cA$ be FCA monotone w.r.t.\ sets with $\cE=2^{[n]}$. Let $\mu$ be a symmetric measure, i.e., $\mu(F)=\mu(G)$ for $F,G\in 2^{[n]}$ such that $|F|=|G|$ (see e.g.~\cite{MirandaGrabischGil2002}). Let us set $\mu^i:=\mu(F)$, if $|F|=i$, $i\in[n]_0$. Then the generalized survival function w.r.t.\ $\mu$ takes the form $$\mu_\cA({\x},{\alpha})=\sum_{i=0}^n \mu^i\bin_{\big[\min\limits_{|E|=n-i}\aA[E]{\x},\min\limits_{|E|=n-i+1}\aA[E]{\x}\big)}(\alpha).$$

\item Let $\cA$ be FCA monotone w.r.t.\ sets with $\cE=2^{[n]}$. Let $\Pi$ be a possibility measure given for any $F\subseteq[n]$ as $\Pi(F)=\max_{i\in F}\pi(i)$. 
The function $\pi\colon [n] \to [0, 1]$, $\pi(i) =\Pi (\{i\})$ is called a possibility distribution (of $\Pi$), see e.g.~\cite{ZADEH19999}. Let $\sigma\colon[n]\to[n]$ be a permutation such that $0=\pi(\sigma(0))\leq\pi(\sigma(1))\leq\dots\leq\pi(\sigma(n))=1$ and $G_{\sigma(i)}=\{\sigma(i),\dots,\sigma(n)\}$ for $i\in\{1,\dots,n\}$ with $G_{\sigma(n+1)}=\emptyset$ and $\aA[G_{\sigma(0)}]{\x}=\infty$. Then the generalized survival function w.r.t.\ $\Pi$ takes the form $$\Pi_{\cA}(\x,\alpha)=\sum_{i=0}^n\pi(\sigma(i))\bin_{\big[\aA[G_{\sigma(i+1)}]{\x},\aA[G_{\sigma(i)}]{\x}\big)}(\alpha).$$

\item Let $\cA$ be FCA monotone w.r.t.\ sets with $\cE=2^{[n]}$. Let $\mathrm{N}$ be a necessity measure given for any $F\subseteq[n]$ as $\mathrm{N}(F)=1-\max_{i\notin F} \pi(i)$ with the convention that the maximum of empty set is $0$. Let  $\sigma\colon[n]\to[n]$ be a permutation given as in the previous case. Then the generalized survival function w.r.t.\ $N$ takes the form $$N_{\cA}(\x,\alpha)=\sum_{i=0}^n\big(1-\pi(\sigma(i))\big)\bin_{\big[\min\limits_{k\geq i}\aA[\{\sigma(k)\} ]{\x},\min\limits_{k> i}\aA[\{\sigma(k)\}]{\x}\big)}(\alpha),$$ where we need the convention $\min\limits_{k\geq 0}\aA[\{\sigma(k)\} ]{\x}=0$.

\end{enumerate}
\end{corollary}


Let us compare both approaches to calculate the generalized survival function. Comparing Example~\ref{priklad1} and Example~\ref{priklad2} one can observe that the domain partition of the generalized survival function (the interval $[0,\infty)$) in Example~\ref{priklad1} is a superset of the domain partition of the same generalized survival function in Example~\ref{priklad2}. This observation, we show, holds in general.  As a consequence, we get that the number of nonzero summands in expression~\eqref{vyjgsf2} is less than in expression~\eqref{vyjgsf1}. The proof can be found in Appendix.

\begin{proposition}\label{porovnanie}
Let $\cA$ be a FCA, $\mu\in\mathbf{M}$, $\x\in[0,+\infty)^{[n]}$.
For each $i\in\ozn{\kappa-1}$ there exists $j$ such that $[\sA_i,\sA_{i+1})\subseteq [\min_{k\leq j}\sA_{\inv{k}}, \min_{k< j}\sA_{\inv{k}})$.
\end{proposition}

\section{Permutations and visualization}\label{druhy_sposob}

Both formulas in Section~\ref{prvy_sposob}, presented in Theorem~\ref{zjednodusenie_def}, \ref{gsf2}
ask to minimize the~values of either monotone measure or aggregation operator. 
The~main aim of the~present section is to show that the  computation of the generalized survival function (and the~minimization process) may be accomplished just on a set of integer indices of the~values of monotone measure and aggregation operator.
Moreover, the~whole procedure may be visualized, so the~computation of the~generalized survival function becomes easily accessible. The~main tools are several functions on a~set of indices~$\ozn{\kappa-1}$.

\subsection{Computation via indices}\label{subsec: indices}
Let us introduce the main tool that we shall work with in this subsection. In accordance with~\eqref{EiAi},~\eqref{Fi} one can see that
each set from the~collection~$\cE$ appears in the~basic enumeration $E_*\colon \ozn{\kappa-1}\to\cE$ once and then its complement appears again in the~basic enumeration $F_*\colon \ozn{\kappa-1}\to\hat{\cE}$. More precisely, if $V\in\cE$ then there are unique indices $i,j\in\ozn{\kappa-1}$ such that 
\begin{center}
$V=E_i$ and $V^c=F_j$.    
\end{center}
By going through all the sets from~$\cE$ we can define a permutation~$(\cdot)\colon\ozn{\kappa-1}\to\ozn{\kappa-1}$ which will describe connections among these two indices. More precisely,
the~basic permutation, essential for this subsection, $(\cdot)\colon\ozn{\kappa-1}\to\ozn{\kappa-1}$ is defined by 
\begin{center}
$(i)=j$ whenever $E_i=F_j^c$,    
\end{center}
i.e., $F_{(i)}=E_i^c$. Thus, the~permutation $(\cdot)$ is in accordance with the~previously adopted notation since $\mu_{(i)}=\mu(F_{(i)})=\mu(E_i^c)$. Let us illustrate the computation of the~permutation $(\cdot)$.

\begin{example}
 Let us consider the same input data as in Example~\ref{graphical_representation}.  Then the index space is $\ozn{5}$.
  Thus $(\cdot)$ is a~map from~$\ozn{5}$ to $\ozn{5}$. To calculate the~value $(0)$ we need the~set $E_0=\emptyset$. Its complement is $\{1,2,3\}$, which has index $5$ in the~enumeration~$F_*$, i.e., $F_5=\{1,2,3\}$. Thus $(0)=5$. Similarly, to compute~$(1)$, we see that $E_1=\{3\}$ and $E_1^c=\{1,2\}=F_2$. Thus $(1)=2$. Continuing in a~similar fashion, we obtain all the~values, 
\begin{center}
$(2)=4$, $(3)=1$, $(4)=3$, $(5)=0$.     
\end{center}
\end{example}

The arrangement from previous example may be graphically represented by a~diagram or by graphs, see Figure~\ref{obr1}.
Each of these representations has certain advantages.
In diagram, the~domain of~$(\cdot)$ is the~lower axis, while the~codomain is the~upper one. Thus the~indices of the~aggregation operator are in the~lower axis, the~indices in the~upper axis correspond to the~values of the~monotone measure.  From a practical point of view, the axes are reversed order.  To describe all visualizations deeply, let us point out that the index assignment process can be seen also reciprocally, i.e., let us define $\inv{\cdot}\colon\ozn{\kappa-1}\to\ozn{\kappa-1}$ \begin{center}
$\inv{j}=i$ whenever $E_i^c=F_j$,    
\end{center}
i.e., $E_{\inv{j}}=F_{j}^c$. It is easy to check that $\inv{\cdot}=(\cdot)^{-1}$. 


\begin{figure}[H]
    \centering
     \begin{subfigure}[b]{0.46\textwidth}
         \centering
         \raisebox{0.5\height}{\includegraphics[scale=1.1]{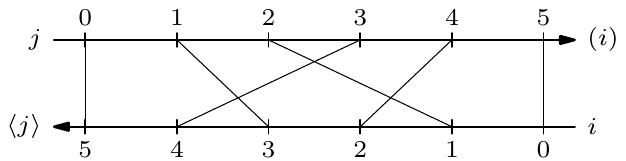}}
         \caption{Diagram of $(\cdot)$, resp.\ $\inv{\cdot}=(\cdot)^{-1}$}
         \label{obr1a}
     \end{subfigure}
     \begin{subfigure}[b]{0.26\textwidth}
         \centering
         \includegraphics[scale=1.1]{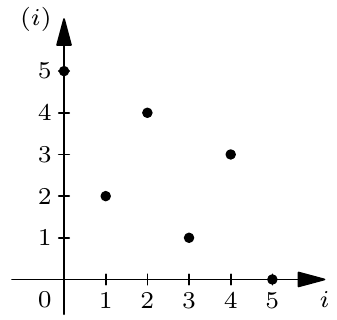}
         \caption{Graph of $(\cdot)$}
         \label{ge_sur_func}
     \end{subfigure}
     \begin{subfigure}[b]{0.26\textwidth}
         \centering
         \includegraphics[scale=1.1]{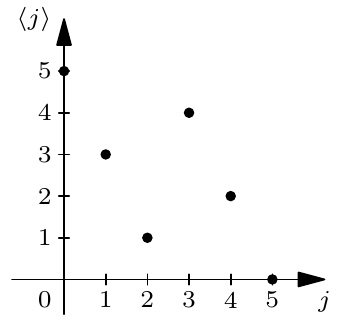}
         \caption{Graph of $\inv{\cdot}=(\cdot)^{-1}$}
         \label{ge_sur_func_2}
    \end{subfigure}
        \caption{Diagram and graphs of functions~$(\cdot)$, $\inv{\cdot}$ for data from Example~\ref{graphical_representation}}
    \label{obr1}
\end{figure}

Once comparing Definition~\ref{def: gsf} of the~generalized survival function and the~map~$(\cdot)$, one can see that the~assignment~$(\cdot)$ is a~part of the~formula defining the~generalized survival function. More precisely, we have 
\[
\gsf{}{\x}{\alpha}=\min\left\{\mu_{(i)}: \sA_i\leq\alpha,\, i\in\ozn{\kappa-1}\right\},\quad\alpha\in[0,+\infty).
\]
Then the~whole computation in the~previous formula may be visualised via Figure~\ref{obr1a}. Indeed, once we need to compute~$\gsf{}{\x}{\alpha}$ for a given $\alpha$, we can find the~largest index~$i$ such that $\sA_i\leq\alpha$. Afterwards we need to consider all the~indices on the~upper axis in
Figure~\ref{obr1a} adjacent with the~indices greater than or equal to~$i$ (the~right-hand side indices with respect to~$i$ on the lower axis). Finally, the~minimization the~values of monotone measure of sets with the~selected indices on the~upper axis leads to the~value $\gsf{}{\x}{\alpha}$. However, the~application of the~assignment~$(\cdot)$ goes far beyond the~latter observations.

To understand the~real contribution of~$(\cdot)$, we need to analyse crossing-overs in Figure~\ref{obr1a}. The crossing-overs can be in only two cases:\begin{itemize}\item if $\textstyle\min_{k\leq i}\mu_{(k)}<\mu_{(i)}$, then the value $\mu_{(i)}$ is not achieved by the generalized survival function because of Theorem~\ref{zjednodusenie_def};
\item if $\mu_{(k)}=\mu_{(l)}$, $k,l\in\ozn{\kappa-1}$, where $k<l$ and $(k)<(l)$, which corresponds to the ambiguity of the arrangement~\eqref{Ei}, or~\eqref{Fi}.
\end{itemize}
These connections can be removed or redefined in an appropriate manner without changing the formula of generalized survival functions. Thus let us redefine the mapping~$(\cdot)$ in a manner that will be beneficial for us. Let us define a~map $\bi\colon\ozn{\kappa-1}\to\ozn{\kappa-1}$ by 
\begin{align}\label{ii}
\bi(i)=\min\{(0),\dots,(i)\}.
\end{align} 
The previous mapping will shorten a lot of further expressions, and calculations. It will be useful in Section~\ref{Choquet} in deriving the generalized Choquet integral formulas.

\begin{figure}[H]
    \begin{center}
    \includegraphics[scale=1.2]{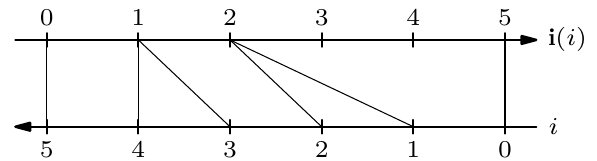}
    \caption{Diagram of the map~$\bi$ for data from Example~\ref{graphical_representation}}\label{obr2}
    \end{center}
\end{figure}
Figure~\ref{obr2} contains a~diagram describing function~$\bi$ computed for data from Example~\ref{graphical_representation}. Although one can use formula~\eqref{ii} directly and compute~$\bi$ algebraically, it can be easily obtained also from the~diagram in Figure~\ref{obr1a}. Indeed, since $(4)=3$ and $(3)=1$, the~corresponding edges in Figure~\ref{obr1a} are crossed-over. It is just enough to eliminate this crossing-over by redefining $\bi(4)=1$. Similarly for crossing-over, which corresponds to values $(1)$ and $(2)$. After its elimination we obtain~$\bi$.

\begin{lema}\label{vl_i}
Let $\cA$ be a FCA, $\mu\in\mathbf{M}$, $\x\in[0,+\infty)^{[n]}$, and 
$\bi$ be a map given in~(\ref{ii}). Then for any $i\in\ozn{\kappa-1}$ it holds
$$\min_{k\leq i} \mu_{(k)}=\mu_{\bi(i)}\text{.}$$
\end{lema}
\proof
Because of the definition of $(\cdot)$ and from the arrangement~(\ref{Fi}) we get equalities 
$$\min_{k\leq i} \mu_{(k)}=\min\{\mu_l:l=(k), k\leq i\}=\mu_{\min\{(0),\dots,(i)\}}=\mu_{\bi(i)}\text{.}$$
Thus we get the required result.
\qed

Via the~assignment~$\bi$, the formula of the~generalized survival function  can be described directly, without the~need to minimize, compare the following corollary with Theorem~\ref{zjednodusenie_def}.

\begin{corollary}\label{application_i}
Let $\cA$ be a FCA, $\mu\in\mathbf{M}$, $\x\in[0,+\infty)^{[n]}$, and 
$\bi$ be a map given in~(\ref{ii}). Then
\begin{align}\label{restated_formula}
\gsf{}{\x}{\alpha}=\sum_{i=0}^{\kappa-1} \mu_{\bi(i)}\bin_{[\sA_i,\sA_{i+1})}(\alpha)
\end{align}
for any $\alpha\in[0,+\infty)$.
\end{corollary}
\proof
The required result follows from Lemma~\ref{vl_i} and Theorem~\ref{zjednodusenie_def}.
\qed



The~difference between Theorem~\ref{zjednodusenie_def} and Theorem~\ref{gsf2} is in the~object being minimized. In Theorem~\ref{zjednodusenie_def} are minimized the~values of monotone measure, while in Theorem~\ref{gsf2} are minimized the~values of aggregation operator. Thus Corollary~\ref{application_i} is the~counterpart of Theorem~\ref{zjednodusenie_def} and naturally, the need to introduce a counterpart of Theorem~\ref{gsf2} arises. 
Similarly, as the formula for the generalized survival function can be rewritten via permutation $(\cdot)$, it can be also rewritten via permutation $\inv{\cdot}$ as follows
\[
\gsf{}{\x}{\alpha}=\min\left\{\mu_j: \sA_\inv{j}\leq\alpha,\, j\in\ozn{\kappa-1}\right\},\quad\alpha\in[0,+\infty).
\] 
%
Let us recall that $\inv{\cdot}=(\cdot)^{-1}\colon\ozn{\kappa-1}\to\ozn{\kappa-1}$, i.e., $\inv{(i)}=i,(\inv{j})=j$, and $E_\inv{j}=F_j^c$, i.e., $F_j=E_\inv{j}^c$. The~permutation $\inv{\cdot}$ is in accordance with the~previously adopted notation since $\sA_\inv{j}=\aA[E_\inv{j}]{\x}=\aA[F_j^c]{\x}$. Let us define a~map $\bj\colon\ozn{\kappa-1}\to\ozn{\kappa-1}$ as follows
\begin{align}\label{jj}
\bj(j)=\min\{\inv{0},\dots,\inv{j}\}.
\end{align}

\begin{remark}\label{nerastucost_i_j}
It is easy to see that for any $k\in\ozn{\kappa-1}$ it holds $\bi(k)\leq(k)$ and $\bj({k})\leq\inv{{k}}$. 
Moreover, the maps $\bi,\bj$ are nonincreasing.
\end{remark}

Similarly as in the case of the map $\bi$ we can shorten a lot of expressions, and calculations with the map $\bj$.
\begin{figure}[H]
    \begin{center}
    \includegraphics[scale=1.2]{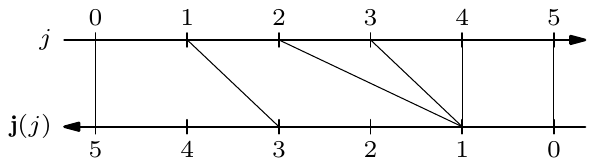}
    \caption{Diagram of the map~$\bj$ for data from Example~\ref{graphical_representation}.}\label{obr3}
    \end{center}
\end{figure}

The~corresponding diagram of~$\bj$ based on data from Example~\ref{graphical_representation} is depicted in Figure~\ref{obr3}. Note that similarly to the~diagram in Figure~\ref{obr2} it does not contain any crossing-over, and may be analogously created either directly using formula~\eqref{jj} or by eliminating crossings-over in the~diagram in Figure~\ref{obr1}. 

\begin{lema}\label{vl_j}
Let $\cA$ be a FCA, $\mu\in\mathbf{M}$, $\x\in[0,+\infty)^{[n]}$, and 
$\bj$ be a map given as in~(\ref{jj}). Then for any $j\in\ozn{\kappa-1}$ it holds
$$\min\limits_{k\leq j}\sA_{\inv{k}}=\sA_{\bj(j)}\text{.}$$
\end{lema}
\proof
Because of the definition of $(\cdot)$ and from the arrangement~(\ref{Ei}) we immediately have
$$\min\limits_{k\leq j}\sA_\inv{k}=\min\{\sA_l: l=\inv{k},k\leq j\}
=\sA_{\min\{\inv{0},\dots,\inv{j}\}} =\sA_{\bj(j)}\text{,}$$
thus we get the required result.
\qed

From the formula~(\ref{vyjgsf2}) we get the new expression of the generalized survival function in~the~following form.

\begin{corollary}\label{application_j}
Let $\cA$ be a FCA, $\mu\in\mathbf{M}$, $\x\in[0,+\infty)^{[n]}$, and 
$\bj$ be a map given as in~(\ref{jj}). Then 
\begin{align}\label{restated_formula_2}
\gsf{}{\x}{\alpha}=\sum_{j=0}^{\kappa-1}\mu_j\bin_{\big[\sA_{\bj(j)},\sA_{\bj(j-1)}\big)}(\alpha)
\end{align}
for any $\alpha\in[0,+\infty)$ with the convention $\sA_{\bj(-1)}=+\infty$. 
\end{corollary}
\proof
The required result follows from a Lemma~\ref{vl_j} and Theorem~\ref{gsf2}.
\qed


Let us conclude this subsection with the~observation that if $(\cdot)$ is decreasing, then formulas~\eqref{vyjgsf1},~\eqref{vyjgsf2} and~\eqref{restated_formula},~\eqref{restated_formula_2} will be simplified.

\begin{corollary}\label{dosledok_formuly_bez_i_a_j}
Let $\cA$ be a FCA, $\mu\in\mathbf{M}$, $\x\in[0,+\infty)^{[n]}$.  If the mapping $(\cdot)$ is decreasing, then 
\begin{align*}
\gsf{}{\x}{\alpha}&=\sum_{i=0}^{\kappa-1}\mu_{\kappa-1-i}\cdot\bin_{\big[\sA_i,\sA_{i+1}\big)}(\alpha)=\sum_{j=0}^{\kappa-1}\mu_j\bin_{\big[\sA_{\kappa-1-j},\sA_{\kappa-j}\big)}(\alpha)
\end{align*}
for any $\alpha\in[0,+\infty)$.
\end{corollary}
\begin{proof}
If the map $(\cdot)$ is decreasing, then $\inv{\cdot}$ is also decreasing and it holds $(\cdot)=\bi$, $\inv{\cdot}=\bj$. Then expressions~\eqref{restated_formula},~\eqref{restated_formula_2}  are simplified
\begin{align*}
\gsf{}{\x}{\alpha}=\sum_{i=0}^{\kappa-1} \mu_{(i)}\bin_{[\sA_i,\sA_{i+1})}(\alpha)=\sum_{j=0}^{\kappa-1}\mu_j\bin_{\big[\sA_{\inv{j}},\sA_{\inv{j-1}}\big)}(\alpha).
\end{align*}
Further, the map $(\cdot)$ is decreasing if and only if for any $i\in\ozn{\kappa-1}$ it holds $(i)=\kappa-1-i$, or  equivalently for any $j\in\ozn{\kappa-1}$ it holds $\inv{j}=\kappa-1-j$, thus $(\cdot)=\inv{\cdot}$. This completes the proof.
\qed
\end{proof}

\begin{example}\label{example_bez_i_a_j}
Let us consider the collection $\cE=\{\emptyset, \{1\},\{3\},\{1,2\},\{1,2,3\}\}$, the family of conditional aggregation operators  $\cA^{\mathrm{sum}}=\{\aAi[E]{\cdot}{\mathrm{sum}}:E\in\cE\}$, the vector $\x=(2,3,4)$, and the monotone measure $\mu$ on $\hat{\cE}$ with values $\mu(\emptyset)=0$, $\mu(\{3\})=0.3$, $\mu(\{1,2\})=0.5$, $\mu(\{2,3\})=0.8$, $\mu(\{1,2,3\})=1$.
\begin{table}[H]
\renewcommand*{\arraystretch}{1.2}
\centering
\begin{tabular}{|n|o|o|o|o|o|}
\hline
$i,j$ & 0 & 1 & 2 & 3 & 4  \\ \hline
$E_i$ & $\emptyset$ & $\{1\}$ & $\{3\}$ & $\{1,2\}$ & $\{1,2,3\}$ \\\hline
$\sA_i$ & $0$ & $2$ & $4$ & $5$ & $9$ \\\hline
$F_j$ & $\emptyset$ & $\{3\}$ & $\{1,2\}$ & $\{2,3\}$ & $\{1,2,3\}$ \\ \hline
$\mu_j$ & $0$ & $0.3$ & $0.5$ & $0.8$ & $1$ \\ \hline
 \end{tabular}
\end{table}
\noindent The diagram of functions $(\cdot)$ and $\inv{\cdot}$ can be seen in the Figure~\ref{obr_diagram_bez_i_j}.
\begin{figure}[h]
    \begin{center}
    \includegraphics[scale=1.2]{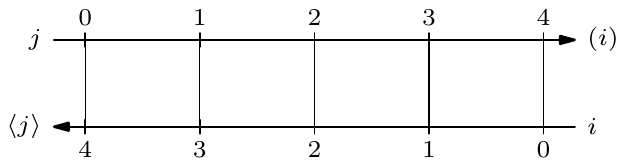}
    \caption{Diagram of $(\cdot)$ and $\inv{\cdot}$ from Example~\ref{example_bez_i_a_j}}\label{obr_diagram_bez_i_j}
    \end{center}
\end{figure}

\noindent It is easy to verify that the assumptions of Corollary~\ref{dosledok_formuly_bez_i_a_j} are satisfied.
According to this result, the generalized survival function has the form
$$\gsf{\mathrm{sum}}{\x}{\alpha}=1\cdot\bin_{[0,2)}(\alpha)+0.8\cdot\bin_{[2,4)}(\alpha)+0.5\cdot\bin_{[4,5)}(\alpha)+0.3\cdot\bin_{[5,9)}(\alpha)$$
for any $\alpha\in[0,+\infty)$.
\end{example}

From the previous one can see that the assumption of Corollary~\ref{dosledok_formuly_bez_i_a_j} can be graphically interpreted in the way that there are no crossing-overs in the diagrams, see Figure~\ref{obr_diagram_bez_i_j}.

In the following, we apply the~obtained formulas to the~solution of the~Knapsack problem described in~Section~\ref{sec: motivation}. 

\paragraph{Solution of the Knapsack problem}
Let us assume that a total of $800$\,ml of liquids are already packed in the knapsack.
Thus, it is still possible to pack $200$\,ml of liquids. Let us choose from the products listed in Table~\ref{batozina}. In the table, except for the volume of the products also their price is indicated.
\begin{table}[H]
    \centering
    \begin{tabular}{c|c|c|c|c}
         product & $a$ & $b$ & $c$ & $d$ \\\hline
         volume & $80$ & $75$ & $55$ & $65$ \\
         price & $1.2$ & $1$ & $0.6$ & $0.8$
    \end{tabular}
    \caption{List of products}
    \label{batozina}
\end{table}
\noindent As we have pointed out in Section~\ref{sec: motivation}, our goal is to minimize purchase costs for products that we shall need to buy at the holiday destination (we suppose prices are higher here). Moreover, we have to keep a limit of $200$\,ml of liquids we can carry, i.e., this is the limit for the total volume of products, we shall buy at home. We shall solve the optimization problem
$$\min\{\mu(E^c): \aAi[E]{\x}{\text{sum}}\leq200,\,E\in\collectionGSF\}\text{,}$$
with $\x=(80,75,55,65)$, the collection $\collectionGSF$ consists of all possible combinations of elements $a,b,c,d$, and the monotone measure $\mu$ represents the price of products. The values of $\mu$ together with the values of the conditional aggregation operator $\sA^{\mathrm{sum}}$ can be seen in Table~\ref{tabulka_batozina}.
\begin{table}[H]
\renewcommand*{\arraystretch}{1.2}
    \centering
    \small
    \begin{tabular}{c|c|c|c|c|ccc|c|c|c|c|c}
         $\mu_j$ & $\mu(E^c)$ & $E^c$ & $E$ & $\sA^\text{sum}$ & $\sA_i$ &\hspace*{0.3cm}& $\mu_j$ & $\mu(E^c)$ & $E^c$ & $E$ & $\sA^\text{sum}$ & $\sA_i$\\\cline{1-6}\cline{8-13}
         $\mu_{15}$ & $3.6$ & $\{a,b,c,d\}$ & $\emptyset$ & $0$ & $\sA_0$ && $\mu_9$ & $1.8$ & $\{a,d\}$ & $\{b,c\}$ & $130$ & $\sA_6$\\
         $\mu_{11}$ & $2.4$ & $\{b,c,d\}$ & $\{a\}$ & $80$ & $\sA_4$ && $\mu_7$ & $1.8$ & $\{a,c\}$ & $\{b,d\}$ & $140$ & $\sA_8$\\
         $\mu_{13}$ & $2.5$ & $\{a,c,d\}$ & $\{b\}$ & $75$ & $\sA_3$ && $\mu_{10}$ & $2.2$ & $\{a,b\}$ & $\{c,d\}$ & $120$ & $\sA_5$\\
         $\mu_{14}$ & $3.0$ & $\{a,b,d\}$ & $\{c\}$ & $55$ & $\sA_1$ && $\mu_2$ & $0.8$ & $\{d\}$ & $\{a,b,c\}$ & $210$ & $\sA_{13}$\\
         $\mu_{12}$ & $2.5$ & $\{a,b,c\}$ & $\{d\}$ & $65$ & $\sA_2$ && $\mu_1$ & $0.6$ & $\{c\}$ & $\{a,b,d\}$ & $220$ & $\sA_{14}$\\
         $\mu_5$ & $1.4$ & $\{c,d\}$ & $\{a,b\}$ & $155$ & $\sA_{10}$ && $\mu_3$ & $1.0$ & $\{b\}$ & $\{a,c,d\}$ & $200$ & $\sA_{12}$\\
         $\mu_8$ & $1.8$ & $\{b,d\}$ & $\{a,c\}$ & $135$ & $\sA_7$ && $\mu_4$ & $1.2$ & $\{a\}$ & $\{b,c,d\}$ & $195$ & $\sA_{11}$\\
         $\mu_6$ & $1.4$ & $\{b,c\}$ & $\{a,d\}$ & $145$ & $\sA_9$ && $\mu_0$ & $0$ & $\emptyset$ & $\{a,b,c,d\}$ & $275$ & $\sA_{15}$\\
    \end{tabular}
    \caption{Values of the conditional aggregation operator and corresponding monotone measure}
    \label{tabulka_batozina}
\end{table}
We can notice that in Table~\ref{tabulka_batozina} there are $15$ different values of the conditional aggregation operator and $12$ different values of the (nonadditive) monotone measure. To solve the knapsack problem we have to find the value of the generalized survival function at point $200$. This can be easily done using formula~\eqref{vyjgsf1} or~\eqref{restated_formula}. The value of $200$ lies in the interval $[\sA_{12},\sA_{13})$, where the generalized survival function takes the value
$$\min_{k\leq 12}\mu_{(k)}=\mu_{\bi(12)}=\mu_3=1\text{.}$$
The same result we get when we determine the whole formula of the generalized survival function using 
formula~\eqref{skrateniegsf2f}
\begin{align*}
   \gsf{\mathrm{sum}}{\x}{\alpha}&=3.6\cdot\mathbf{1}_{[0;55)}(\alpha)+3\cdot\mathbf{1}_{[55;65)}(\alpha)+2.5\cdot\mathbf{1}_{[65;80)}(\alpha)+2.4\cdot\mathbf{1}_{[80;120)}(\alpha)\\
   &+2.2\cdot\mathbf{1}_{[120;130)}(\alpha)+1.8\cdot\mathbf{1}_{[130;145)}(\alpha)+1.4\cdot\mathbf{1}_{[145;195)}(\alpha)+1.2\cdot\mathbf{1}_{[195;200)}(\alpha)\\
   &+\mathbf{1}_{[200;210)}(\alpha)+0.8\cdot\mathbf{1}_{[210;220)}(\alpha)+0.6\cdot\mathbf{1}_{[220;275)}(\alpha)\text{.}
\end{align*}

\noindent As we can see, the result coincides with the result from the previous calculation. So, a traveler should buy products $a,c,d$ at home and product $b$ at the destination.

\subsection{How to draw a~graph of a~generalized survival function}

In the previous section, we have presented visualizations of maps $(\cdot)$ and $\inv{\cdot}$ using both a diagram and a graph in the~Cartesian coordinate system, see demonstratory Figure~\ref{obr1} based on inputs from~Example~\ref{graphical_representation}.
We have already pointed out the advantages of the first visualization, and in the following we deal with the second one.
In fact, we shall show how the~graph of $(\cdot)$ or $\inv{\cdot}$ can be transformed to the~plot of the~generalized survival function.

The~first step is to transform the~graphs of $(\cdot)$ or $\inv{\cdot}$, see Figure~\ref{ge_sur_func}\subref{ge_sur_func_2}, into the~graphs of the~mappings $\bi$ or $\bj$, see Figure~\ref{image_i(i)_and_j(j)}, decreasing some of the~values of $(\cdot)$ and $\inv{\cdot}$ to produce nonincreasing functions $\bi$ and $\bj$, respectively. More precisely, the~formulas~\eqref{ii} and~\eqref{jj} are used to compute $\bi$ and $\bj$ from $(\cdot)$ and $\inv{\cdot}$, respectively.

\begin{figure}[H]
    \begin{center}
    \includegraphics[scale=1.2]{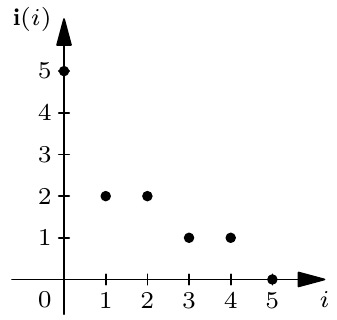}\hspace{50pt}
    \includegraphics[scale=1.2]{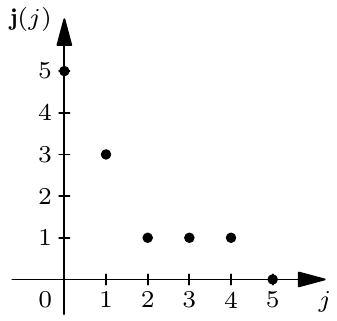}
    \caption{Graphs of $\bi$ and $\bj$.}\label{image_i(i)_and_j(j)}
    \end{center}
\end{figure}

The~second step is to extend the~domain of~$\bi$ from $\ozn{\kappa-1}$ to $[0,+\infty)$, and obtain the~graph of the~\textit{indexed generalized survival function} $\gsf{I}{\x}{\beta}$. It is enough to naturally define $\gsf{I}{\x}{\beta}=\bi(\lfloor\beta\rfloor)$ for $\beta<\kappa-1$ and $\gsf{I}{\x}{\beta}=\bi(\kappa-1)$ otherwise, as is done in~Figure~\ref{indexed_survival_function}(a). There is similar way to obtain the~indexed generalized survival function from the~graph of~$\bj$, depicted in~Figure~\ref{indexed_survival_function}(b), which we describe later.

\begin{figure}[H]
     \centering
     \begin{subfigure}[b]{0.49\textwidth}
         \centering
          \includegraphics[scale=1.1]{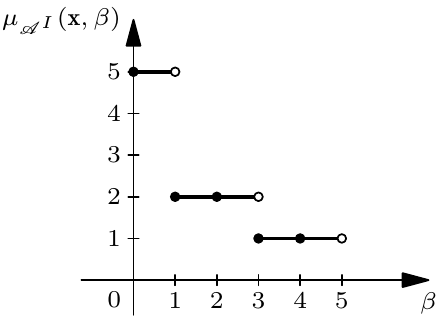}
         \caption{Indexed generalized survival function (using $\bi$)}
         \label{ind_gen_sur_func}
     \end{subfigure}
     \begin{subfigure}[b]{0.49\textwidth}
         \centering
         \includegraphics[scale=1.1]{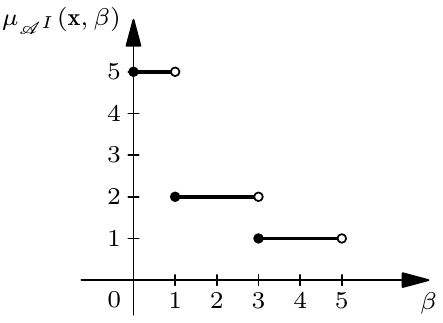}
         \caption{Indexed generalized survival function (using $\bj$)}
         \label{ind_gen_sur_func_2}
     \end{subfigure}
        \caption{Indexed generalized survival functions deriving by using maps $\bi$ and $\bj$}
        \label{indexed_survival_function}
\end{figure}

Before we proceed to the~last step, let us stress that Figure~\ref{indexed_survival_function} of~$\gsf{I}{\x}{\beta}$ is very close to how the~real generalized survival function $\gsf{}{\x}{\alpha}$ depicted in Figure~\ref{priradenie} looks like. Vaguely speaking, the~only difference is that the~graph of~$\gsf{}{\x}{\alpha}$ has values $\sA_0,\sA_1,\dots,\sA_{\kappa-1}$ on the~horizontal axis instead of values $0,1,\dots,\kappa-1$, and values $\mu_0,\mu_1,\dots,\mu_{\kappa-1}$ on the~vertical axis again instead of values $0,1,\dots,\kappa-1$. Moreover, some values $\sA_i,\sA_j$ and $\mu_i,\mu_j$ may coincide even for different~$i,j$.

\begin{figure}[H]
    \centering
    \includegraphics[scale=1.1]{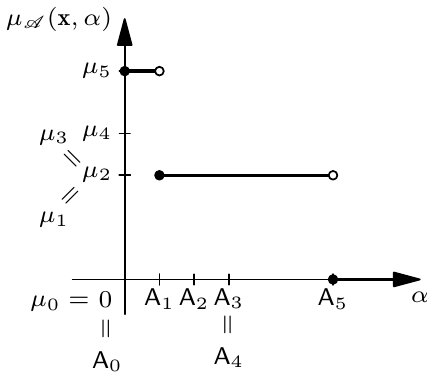}
    \caption{Generalized survival function}
    \label{priradenie}
\end{figure}

More formally, the~similarities between $\gsf{I}{\x}{\beta}$ and $\gsf{}{\x}{\alpha}$ may be understood once we represent the~latter one via the~formula~\eqref{restated_formula}, and the~first one in a~similar way, namely
\begin{align}\label{index_formula_1}
\gsf{}{\x}{\alpha}=\sum_{i=0}^{\kappa-1} \mu_{\bi(i)}\bin_{[\sA_i,\sA_{i+1})}(\alpha),\hspace{1cm}\gsf{I}{\x}{\beta}=\sum_{i=0}^{\kappa-2} \bi(i)\bin_{[i,\,i+1)}(\beta). 
\end{align}
The~equalities~\eqref{index_formula_1} are based on~function~$\bi$, but we may proceed similarly with function~$\bj$ using the~equality~\eqref{restated_formula_2}:
\begin{align}\label{index_formula_2}
\gsf{}{\x}{\alpha}=\sum_{i=0}^{\kappa-1} \mu_{\bi(i)}\bin_{[\sA_i,\sA_{i+1})}(\alpha),\hspace{1cm}\gsf{I}{\x}{\beta}&=\sum_{j=1}^{\kappa-1} j\bin_{[\bj(j),\,\bj(j-1))}(\beta). 
\end{align}
 The~equality~\eqref{index_formula_2} gives clues how to obtain the~graph of the~indexed generalized survival function~$\gsf{I}{\x}{\beta}$ from the~graph of the~map~$\bj$. There is necessary one more step, namely, one has to draw a~graph of the~generalized inverse~~$\bj^-$ from the~graph of~$\bj$ first. Then, similarly to the~case of~$\bi$, define $\gsf{I}{\x}{\beta}=\bj^-(\lfloor\beta\rfloor)$ for $\beta<\kappa-1$ and $\gsf{I}{\x}{\beta}=\bj^-(\kappa-1)$ otherwise, see~Figure~\ref{indexed_survival_function}(b).

Let us summarize the calculation of the indexed generalized survival function.
First, we construct the bijections $E_*$ a $F_*$ as it shown in \eqref{Ei} a \eqref{Fi}.
Then, using the permutation $(\cdot)\colon \ozn{\kappa-1}\to\ozn{\kappa-1}$ described above, we obtain the values of $\bi(i)$ and $\bj(i)$ for each $i\in\ozn{\kappa-1}$, and thus we can construct the indexed generalized survival function.
Finally, for each index $k\in\ozn{\kappa-1}$ on the horizotal, resp.\ vertical, axes we assign the value $\sA_k$, resp.\ $\mu_k$.
This entire process can be described by the following Algorithm~\ref{alg_GSF_from_i} and Algorithm~\ref{alg_GSF_from_j}.
Note that these algorithms assume the fact that the values of map $(\cdot)$, resp.\ $\inv{\cdot}$ are known to the user, or for completeness, we present  algorithms for its calculation in Appendix.

\bigskip
\noindent%
\begin{minipage}{0.49\textwidth}
\SetKwRepeat{Do}{do}{while}
\begin{algorithm}[H]
\Fn{the-graph-of-GSF$(\collectionGSF, \mu, \x, \cA)$}{
$(0),\dots,(\kappa-1)$ $\leftarrow$ the-$(\cdot)$-map$(\collectionGSF, \mu, \x, \cA)$\;
$GSF:=0$\;
\For{$(i=0$, $i<\kappa-1$, $i\mathrm{++})$}{%
    \If{$(i+1)>(i)$}{%
        $(i+1):=(i)$\;
    }
    $GSF:=GSF+\mu_{(i)}\cdot\bin_{[\sA_i,\sA_{i+1})}$\;
}
\Return{$GSF$\;
}}
\caption{Calculation of generalized survival function (GSF) using the map $\bi$
}
\label{alg_GSF_from_i}
\end{algorithm}
\end{minipage}
\hfill
\begin{minipage}{0.49\textwidth}
\SetKwRepeat{Do}{do}{while}
\begin{algorithm}[H]
\Fn{the-graph-of-GSF$(\collectionGSF, \mu, \x, \cA)$}{
$\inv{0},\dots,\inv{\kappa-1}$ $\leftarrow$ the-$\inv{\cdot}$-map$(\collectionGSF, \mu, \x, \cA)$\;
$GSF:=0$\;
\For{$(j=1$, $j<\kappa$, $j\mathrm{++})$}{%
    \If{$\inv{j-1}<\inv{j}$}{%
        $\inv{j}:=\inv{j-1}$\;
    }
    $GSF:=GSF+\mu_j\cdot\bin_{[\sA_\inv{j},\sA_\inv{j-1})}$\;
}
\Return{$GSF$\;
}}
\caption{Calculation of generalized survival function (GSF) using the map $\bj$
}
\label{alg_GSF_from_j}
\end{algorithm}
\end{minipage}

\bigskip
Note that although the maps $\bi$ and $\bj$ is not appear explicitly in the algorithms, their generation is hidden in their for loops.
In this part of algorithms, the indexed generalized survival function is also generated, while its standard version is immediately assigned to it. 
A zero value of generalized survival function can also be marked in the graph (Figure~\ref{priradenie}), which corresponds to Remark~\ref{gsf=0} and Remark~\ref{A1c}, but it is not necessary, since this value does not bring any new information.

\begin{remark}
The following relationships can be observed from formulas~\eqref{index_formula_1} and~\eqref{index_formula_2}.
Let $I=\{\bi(i):i\in\ozn{\kappa-1}\}$ and $J=\{\bj(j):j\in\ozn{\kappa-1}\}$, $j\in I$ and $i\in J$, then
\begin{align*}
\min\{k\in\ozn{\kappa-1}:\bj(k)=i\}&=\bi(i)\text{,}\\
\min\{k\in\ozn{\kappa-1}:\bi(k)=j\}&=\bj(j)\text{.}
\end{align*}
These equalities hold because of 
\begin{align*}
    \min\{k\in\ozn{\kappa-1}:\bj(k)=i\}&=\min\{k\in\ozn{\kappa-1}:\min\{\inv{0},\dots,\inv{k}\}=i\}\\&=\min\{k\in\ozn{\kappa-1}: \inv{k}=i\}=\min\{k\in\ozn{\kappa-1}: k=(i)\}\\&=\min\{(0),\dots,(i)\}=\bi(i)\text{.}
\end{align*}
The penultimate equality follows from the fact that $(\cdot)$ is nonincreasing, see Remark~\ref{nerastucost_i_j}.
Similarly for the second formula.
\end{remark}

\begin{remark}
It is also possible to use \eqref{index_formula_1} and~\eqref{index_formula_2}  with analogy to the formula \eqref{predpis gsf} for calculating the indexed generalized survival function as follows
\begin{align*}
\gsf{I}{\x}{\beta}&=\min\left\{(i): i\leq \beta,\, i\in\ozn{\kappa-1}\right\}=\bi(\beta)\\
&=\min\left\{j: \inv{j}\leq \beta,\, j\in\ozn{\kappa-1}\right\}=\bj^-(\beta)
\end{align*}
for any $\beta\in[0,+\infty)$ with respect to extended domain of $\bi$ and $\bj^-$, i.e.\ $[0,+\infty)$, described above.
\end{remark}

\section{Generalized Choquet integral computation}\label{Choquet}

The survival function measuring the $\alpha$-level set $\{f>\alpha\} $\footnote{$\{f>\alpha\}=\{x\in X:f(x)>\alpha\}$} of an arbitrary measurable function $f$ via monotone measure $\mu$ is the essence of the concept of the famous Choquet integral~\cite{Choquet1954} defined by $$\mathrm{C}(f,\mu)=\int_{0}^{\infty}\mu(\{f>\alpha\})\,\mathrm{d}\alpha.$$
On the finite spaces, the evaluation formula of the Choquet integral has the simplified form \begin{align}\label{Chformula}\mathrm{C}(\x,\mu)=\sum_{i=1}^{n} \mu{(G_{\sigma(i)})}(x_{\sigma{(i)}}-x_{\sigma{(i-1)}})=\sum_{i=1}^{n} x_{\sigma{(i)}}(\mu{(G_{\sigma(i)})}-\mu{(G_{\sigma(i+1)})}),\end{align} where $\x=(x_{\sigma(1)},x_{\sigma(2)},\dots,x_{\sigma(n)})$, with $\sigma\colon [n]\to[n]$ being a permutation such that $x_{\sigma(1)}\leq x_{\sigma(2)}\leq \dots\leq x_{\sigma(n)}$ with the convention $x_{\sigma(0)}=0$,
$G_{\sigma{(i)}}=\{\sigma{(i)},\dots,\sigma{(n)}\}$ for $i\in[n]$, and $G_{\sigma(n+1)}=\emptyset$. 

Based on generalized survival function, a new concept generalizing the Choquet integral naturally arises. 
In this section, we aim to provide formulas for discrete $\cA$-Choquet integral.

\begin{definition}\label{def: Chintegral}\rm(cf.~\cite[Definition 5.4.]{BoczekHalcinovaHutnikKaluszka2020})
Let $\cA$ be a FCA, $\mu\in\mathbf{M}$ and $\x\in[0,+\infty)^{[n]}$. The Choquet integral with respect to $\cA$ and $\mu$ ($\cA$-Choquet integral, for short) of $\x$ is defined as \begin{align}\label{Chint}\mathrm{C}_{\cA}(\x,\mu)=\int_{0}^{\infty}\gsf{}{\x}{\alpha}\,\mathrm{d}\alpha.\end{align}
\end{definition}

As we have obtained the generalized survival function expressions, see Sections~\ref{prvy_sposob},~\ref{druhy_sposob}, the computation of the $\cA$-Choquet integral becomes a trivial matter. Note that the formulas in the first two lines are obtained by means of mappings $\bi$ and $\bj$, respectively. By introducing these mappings, we managed to obtain expressions similar to the original one in~(\ref{Chformula}).

\begin{theorem}\label{vypocetCh}
Let $\cA$ be a FCA, $\mu\in\mathbf{M}$ and $\x\in[0,+\infty)^{[n]}$. Then
\begin{enumerate}[\rm (i)]
\item $\mathrm{C}_{\cA}(\x,\mu)=\sum_{i=0}^{\kappa-2}\mu_{\bi(i)}(\sA_{i+1}-\sA_i)=\sum_{i=0}^{\kappa-1}\sA_{i+1}\big(\mu_{\bi(i)}-\mu_{\bi(i+1)}\big)$\text{,}\,\hfill{\rm [cf. Corollary~\ref{application_i}]}\\
\item $\mathrm{C}_{\cA}(\x,\mu)=\sum_{i=1}^{\kappa-1}\mu_i\big(\sA_{\bj(i-1)}-\sA_{\bj(i)}\big)=\sum_{i=1}^{\kappa-1}\sA_{\bj(i-1)}(\mu_{i}-\mu_{i-1})$\text{,}\,\hfill{\rm [cf. Corollary~\ref{application_j}]}\\
\item $\mathrm{C}_{\cA}(\x,\mu)=\sum_{i=0}^{\kappa-2}\min_{k\leq i}\mu_{(k)}(\sA_{i+1}-\sA_i)=\sum_{i=0}^{\kappa-1}\sA_{i+1}\big(\min_{k\leq i}\mu_{(k)}-\min_{k\leq i+1} \mu_{(k)}\big)$\text{,}\,\hfill{\rm [cf. Theorem~\ref{zjednodusenie_def}]}\\

\item $\mathrm{C}_{\cA}(\x,\mu)=\sum_{i=1}^{\kappa-1}\mu_i\big(\min_{k<i} \sA_\inv{k}-\min_{k\leq i}\sA_\inv{k}\big)=\sum_{i=1}^{\kappa-1}\min_{k<i}\sA_\inv{k}(\mu_{i}-\mu_{i-1})$\text{.}\,\hfill{\rm [cf. Theorem~\ref{gsf2}]}\\




\end{enumerate}
\end{theorem}

Let us recall that in Corollary~\ref{specimiery} we listed computational formulas for generalized survival function for special measures. Thus, Corollary~\ref{specimiery} helps us to substantially improve the~formulas presented in Theorem~\ref{vypocetCh} for special measures. 



\begin{corollary}\label{specimieryCh}
Let $x\in[0,+\infty)^{[n]}$.
\begin{enumerate}[\rm (i)]
\item Let $\cA$ be FCA and $\bar{\mu}$ be a greatest monotone measure. Then $$\mathrm{C}_{\cA}(\x,\mmu)=\aA[{[n]}]{\x}.$$ 
\item Let $\cA$ be FCA. Let $\mmu$ be a weakest monotone measure. Then $$\mathrm{C}_{\cA}(\x,\bar{\mu})=\min_{E\neq\emptyset}\aA[E]{\x}.$$
\item Let $\cA$ be FCA monotone w.r.t.\ sets with $\cE=2^{[n]}$. Let $\mu$ be a symmetric measure and set $\mu^i:=\mu(F)$, if $|F|=i$, $i\in[n]\cup\{0\}$. Then $$\mathrm{C}_{\cA}(\x,\mu)=\sum_{i=1}^n (\mu^i-\mu^{i-1})\min\limits_{|E|=n-i+1}\aA[E]{\x}.$$
\item Let $\cA$ be FCA monotone w.r.t.\ sets with $\cE=2^{[n]}$. Let $\Pi$ be a possibility measure. Let $G_{\sigma(i)}=\{\sigma(i),\dots,\sigma(n)\}$ for $i\in\{1,\dots,n\}$, where $\sigma$ is a permutation as in Corollary~\ref{specimiery}. Then $$\mathrm{C}_{\cA}(\x,\Pi)=\sum_{i=1}^n\left(\pi(\sigma(i))-\pi(\sigma(i-1))\right)\aA[G_{\sigma(i)}]{\x}.$$
\item Let $\cA$ be FCA monotone w.r.t.\ sets with $\cE=2^{[n]}$. Let $\mathrm{N}$ be a neccessity measure and $\sigma$ be a permutation as in Corollary~\ref{specimiery}. Then $$\mathrm{C}_{\cA}(\x,N)=\sum_{i=1}^n\big(\pi(\sigma(i))-\pi(\sigma(i-1))\big)
\min\limits_{k\geq i}\aA[\{\sigma(k)\}]{\x}.$$


\end{enumerate}
\end{corollary}




\begin{remark} In the following, we aim to emphasize that the $\cA$-Choquet integral covers the famuous Choquet integral.
Let us consider a special family $\cA^{\mathrm{max}}$. 
Taking symmetric measures we obtain $$\mathrm{C}_{\cA^\mathrm{max}}(\x,\mu)=\sum_{i=1}^n (\mu^i-\mu^{i-1})x_{\sigma(n-i+1)}=\sum_{i=1}^n(\mu^{n-i+1}-\mu^{n-i})x_{\sigma(i)},$$ where $\sigma\colon [n]\to[n]$ is a permutation such that $x_{\sigma(1)}\leq\dots\leq x_{\sigma(n)}$ with the convention $x_{\sigma(0)}=0$. This is a formula of Yager's ordered weighted averaging (OWA) operator~\cite{Yager1998} 
as Grabisch has shown,  see~\cite{GRABISCH20111}. Therefore, the generalized $\cA$-Choquet integral w.r.t. symmetric measures can be seen as a new type of the OWA operator.

The formula for possibility measure in Corollary~\ref{specimieryCh}\,(iv) simplifies to the form \begin{align*}\mathrm{C}_{\cA^\mathrm{max}}(\x,\Pi)
=\sum_{i=1}^n(\pi(\sigma(i))-\pi(\sigma(i-1)))\max_{k\geq i}x_{\sigma(k)}
=\sum_{i=1}^n(\pi(\sigma(i))-\pi(\sigma(i-1)))\max_{\pi(k)\geq\pi(\sigma(i))} x_{k},\end{align*} which is the famous Choquet integral with respect to possibility measure, cf.~\cite{DuboisRico2016}, and permutation $\sigma$ as in Corollary~\ref{specimiery}. Similarly, taking necessity measure in~Corollary~\ref{specimieryCh}\,(v) we obtain \begin{align*}\mathrm{C}_{\cA^\mathrm{max}}(\x,\mathrm{N})=\sum_{i=1}^n(\pi(\sigma(i))-\pi(\sigma(i-1)))\min_{k\geq i}x_{\sigma(k)}
=\sum_{i=1}^n(\pi(\sigma(i))-\pi(\sigma(i-1)))\min_{\pi(k)\geq\pi(\sigma(i))} x_{k}.\end{align*}\end{remark}

Due to formulas derived in this section, we are ready to solve the~problem of accomodation options introduced in Section~\ref{sec: motivation}, using the~concept of the generalized Choquet integral.

\paragraph{Solution of the Problem of accomodation options}
The aim of this example is to determine the most suitable accommodation for each character of the traveler (Anthony, Brittany, Charley). We use two methods: the standard and the generalized Choquet integral. Since we are deciding with respect to several criteria, the problem of choosing the best accommodation is a multi-criteria decision-making problem.

As we can see in Table~\ref{options_booking}, the criteria of the multicriteria decision-making process are different types: distance and price are minimization criteria, while the review is maximization criteria. Thus values in tables in the recent form are not suitable for working with, we have to reevaluate them using the standard methods, see~\cite{Singh2020}.

\noindent According to the previous let us adjust the column corresponding to distance:
\begin{itemize}
\item In each table let us choose the minimum value in the column for the distance. 
\item Let us divide the minimum from the previous step by each value from the distance column.
\end{itemize}
Let us repeat this process for the price as well. 
For the reviews column let us divide each value by the maximum value from the corresponding column. 
This gives us the following input vectors for decision-making process $\vectorA_1,\vectorA_2,\vectorB_1,\vectorB_2,\vectorC_1,\vectorC_2$:
\begin{table}[H]
    \centering
    \begin{subtable}{0.32\linewidth}
    \centering
    \begin{tabular}{M|M|M|M}
             & $\text{D}$ & $\text{P}$ & $\text{R}$ \\\hline
       $\vectorA_1$ & $1$ & $0.84$ & $0.875$\\
       $\vectorA_2$ & $0.4$ & $1$ & $1$
    \end{tabular}
    \caption{Anthony}
    \end{subtable}
    \begin{subtable}{0.32\linewidth}
    \centering
    \begin{tabular}{M|M|M|M}
             & $\text{D}$ & $\text{P}$ & $\text{R}$ \\\hline
       $\vectorB_1$ & $1$ & $0.8$ & $0.2$\\
       $\vectorB_2$ & $0.7$ & $1$ & $1$
    \end{tabular}
    \caption{Brittany}
    \end{subtable}
    \begin{subtable}{0.32\linewidth}
    \centering
    \begin{tabular}{M|M|M|M}
             & $\text{D}$ & $\text{P}$ & $\text{R}$ \\\hline
       $\vectorC_1$ & $1$ & $0.95$ & $0.8$\\
       $\vectorC_2$ & $0.2$ & $1$ & $1$
    \end{tabular}
    \caption{Charley}
    \end{subtable}
    \caption{Accommodation options for people -- adjusted data}
    \label{options_booking_adjust}
\end{table}
To evaluate each input vector (alternative) by Choquet integrals we have to know the values of monotone measures of each set $E\subseteq\{D,P,R\}$. 
We already know the values of the monotone measures $\mu$, $\nu$, $\xi$, that specify the characters of persons, for singletons, see Table~\ref{charaktery}. Further, because of the definition of the (normalized) monotone measure, 
\begin{center}$\mu(\emptyset)=\nu(\emptyset)=\xi(\emptyset)=0$,\,\, $\mu(\{\text{D}, \text{P}, \text{R}\})=\nu(\{\text{D}, \text{P}, \text{R}\})=\xi(\{\text{D}, \text{P}, \text{R}\})=1$.\end{center} It remains to determine the values of monotone measures of sets $\{\text{D}, \text{P}\}$, $\{\text{D}, \text{R}\}$, $\{\text{P}, \text{R}\}$.
The question is: How to set this values?
Let us rephrase this question: If we set the values of the monotone measure, how do we know that they are set well?
Various approaches to verify the correctness of the set values are known in the literature. 
Let us mention e.g.\ interaction index \cite{Grabisch1996}, Banzhaf power index or Shapley value.
We chose the Shapley value and the methodology described in~\cite{Grabisch1996}. 
The Shapley value expresses the significance of a given criterion by aggregating in a certain way the values of the monotone measure of the sets to which the given criterion belongs:
\begin{align}\label{shapley_formula}
    \varphi_\mu(i)=\sum_{A\subseteq [n]\backslash \{i\}} \gamma_{[n]}(A)\cdot\left(\mu(A\cup\{i\})-\mu(A)\right)\text{,}
\end{align}
where $i\in[n]$ is the index associated with the criterion and $\textstyle\gamma_{[n]}(A)=\frac{(n-|A|-1)!\cdot|A|!}{n!}$. 
Shapley value has the property that $\textstyle\sum_{i=1}^n\varphi_\mu(i)=1$.
So, we set the monotone measures $\mu, \nu, \xi$ to obtain $$\mu(\{i\})=\varphi_{\mu}(i),\quad \nu(\{i\})=\varphi_{\nu}(i),\quad \xi(\{i\})=\varphi_{\xi}(i),$$
i.e.\ we form the equations given by~\eqref{shapley_formula} and we calculate the remaining values of the monotone measures (in the table are their approximated values to two decimal places):
\begin{table}[H]
    \centering
    \begin{tabular}{c|o|o|o}
        $E$       & $\mu(E)$ & $\nu(E)$ & $\xi(E)$ \\ \hline
        $\{\text{D},\text{P}\}$ & $0.94$ & $0.85$ & $0.63$  \\
        $\{\text{D},\text{R}\}$ & $0.48$ & $0.76$ & $0.71$  \\
        $\{\text{P},\text{R}\}$ & $0.81$ & $0.59$ & $0.84$
    \end{tabular}
    \caption{Monotone measures calculating by Shapley value}
    \label{monotone_measure_Shapley}
\end{table}
Finally, let us calculate the standard and generalized Choquet integral of vectors $\vectorA_1,\vectorA_2,$ $\vectorB_1,\vectorB_2,\vectorC_1,\vectorC_2$ with respect to monotone measures $\mu$, $\nu$, $\xi$. As
a conditional aggregation operator let us use
the standard Choquet integral
\begin{center}
$\sA^{\mathrm{Ch}}(\x|E)=\mathrm{C}(\x\mathbf{1}_E,m)$,
\end{center}
with $\x\in[0,\infty)^{[3]}$, $E\in\collectionGSF=2^{[3]}$ and monotone measure $m$ (we shall use $\mu$, $\nu$, $\xi$ according to character). Further, $\collectionGSF=2^{[3]}$.
We expect that 
\begin{itemize}
    \item $\vectorA_2$ should be preferred to $\vectorA_1$ with respect to character of Anthony,
    \item  $\vectorB_1$ should be preferred to $\vectorB_2$, because of the character of Brittany, 
    \item $\vectorC_2$ preferred to $\vectorC_1$ for Charley.
    \end{itemize}

The results for both integrals are shown in Table~\ref{results}. Let us look at the highlighted diagonal of Table~\ref{results}. 
In the highlighted cells, the first line, we see the calculation of the standard Choquet integrals. As we can see, we did not receive the expected accommodation option preferences.
In the second line of highlighted cells, the generalized Choquet integral is calculated. 
We see that this method matches our expected preferences.
This indicates its benefits in decision-making processes. 
The generalized Choquet integral in the calculation process takes into account the input vector with respect to various conditional sets. If the collection contains more sets than the cardinality of the basic set, this integral aggregates more input values than when using the standard Choquet integral. Thanks to this, the calculation is more detailed, "smoother" and more accurate.

To be complete, we have calculated the overall scores of various combinations of accommod\-ation--person, see the other cells of Table~\ref{results}. 

\begin{table}
    \centering
    \begin{tabular}{c|c|S|S|S}
         accom.\ opt. & type of integral & Anthony ($\mu$) & Brittany ($\nu$) & Charley ($\xi$) \\ \hline
         \multirow{4}{*}{$\vectorA_1$, $\vectorA_2$} & \multirow{2}{*}{\footnotesize{Choquet integral}} & \cellcolor{gray!10} $\vectorA_1\succ \vectorA_2$ & $\vectorA_1\succ \vectorA_2$ & $\vectorA_2\succ \vectorA_1$ \\
         & & \cellcolor{gray!10} \footnotesize{$0.8943>0.8860$} & \footnotesize{$0.9604>0.7540$} & \footnotesize{$0.9040>0.8899$} \\[5pt]
         & \multirow{2}{*}{\footnotesize{\shortstack{generalized Choquet \\ integral}}} & \cellcolor{gray!10} $\vectorA_2\succ \vectorA_1$ & $\vectorA_1\succ \vectorA_2$ & $\vectorA_2\succ \vectorA_1$ \\
         & & \cellcolor{gray!10} \footnotesize{$0.6857>0.6439$} & \footnotesize{$0.6845>0.4554$} & \footnotesize{$0.6566>0.6161$}\\\hline
         \multirow{4}{*}{$\vectorB_1$, $\vectorB_2$} & \multirow{2}{*}{\footnotesize{Choquet integral}} & $\vectorB_2\succ \vectorB_1$ & \cellcolor{gray!10} $\vectorB_2\succ \vectorB_1$ & $\vectorB_2\succ \vectorB_1$ \\
         & & \footnotesize{$0.9430>0.8240$} & \cellcolor{gray!10} \footnotesize{$0.8770>0.8600$} & \footnotesize{$0.9520>0.6180$} \\[5pt]
         & \multirow{2}{*}{\footnotesize{\shortstack{generalized Choquet \\ integral}}} & $\vectorB_2\succ \vectorB_1$ & \cellcolor{gray!10} $\vectorB_1\succ \vectorB_2$ & $\vectorB_2\succ \vectorB_1$ \\
         & & \footnotesize{$0.7127>0.6087$} & \cellcolor{gray!10} \footnotesize{$0.6393>0.5861$} & \footnotesize{$0.6766>0.3551$}\\\hline
         \multirow{4}{*}{$\vectorC_1$, $\vectorC_2$} & \multirow{2}{*}{\footnotesize{Choquet integral}} & $\vectorC_1\succ \vectorC_2$ & $\vectorC_1\succ \vectorC_2$ & \cellcolor{gray!10} $\vectorC_1\succ \vectorC_2$ \\
         & & \footnotesize{$0.9560>0.8480$} & \footnotesize{$0.9650>0.6720$} & \cellcolor{gray!10} \footnotesize{$0.9045>0.8720$} \\[5pt]
         & \multirow{2}{*}{\footnotesize{\shortstack{generalized Choquet \\ integral}}} & $\vectorC_1\succ \vectorC_2$ & $\vectorC_1\succ \vectorC_2$ & \cellcolor{gray!10} $\vectorC_2\succ \vectorC_1$ \\
         & & \footnotesize{$0.7069>0.6654$} & \footnotesize{$0.6895>0.3532$} & \cellcolor{gray!10} \footnotesize{$0.6433>0.6226$}\\
    \end{tabular}
    \caption{Results of preferences of accommodation options. The preference relation is represented by Choquet integrals as follows:
$\x\succ\y$ ($\x$ is preferred to $\y$) if only if $\mathrm{C}(\x,m)>\mathrm{C}(\y,m)$, resp.\ $\mathrm{C}_{\cA}(\x,m)>\mathrm{C}_{\cA}(\y,m)$, where $m$ is $\mu$, $\nu$ or $\xi$, respectively.}
    \label{results}
\end{table}

\section{Searching optimal intervals, and 
indistinguishibility}\label{optimal_sposob}

The Choquet integral is a basic tool for multicriteria decision making and modeling of decision under risk and uncertainty. 
In~\cite{DuboisRico2016}, Dubois and Rico studied the equality conditions of Choquet integrals of particular input vectors. They considered Choquet integrals with respect to possibility and necessity measures. In~\cite{ChenMesiarLiStupnanova2017}, Chen et al. continued their research with a view to a wider class of so-called universal integrals~\cite{KlementMesiarPap2010}. Universal integrals form one class of utility functions in multicriteria decision making. 

In this section, we formulate the equality conditions of generalized survival functions considering arbitrary measures. Naturally, if generalized survival functions coincide, then their Choquet integrals equal. In order to obtain the equality conditions, we study the greatest possible intervals on which the generalized survival function takes its possible values.
For this purpose, for any $j\in\ozn{\kappa-1}$ let us set 
\begin{align}
\label{phii}\fir(j):=\max\{k\in\ozn{\kappa-1}:\mu_k=\mu_j\}\quad\text{and}\quad\fil(j)=\min\{k\in\ozn{\kappa-1}:\mu_k=\mu_j\}.
\end{align}
The following proposition summarizes the basic properties of $\fir$ and $\fil$.

\begin{proposition}\label{vlastnosti_fi} Let $\cA$ be a FCA, $\mu\in\mathbf{M}$, $\x\in[0,\infty)^{[n]}$, $\fir$ and $\fil$ are defined as in~\eqref{phii}  and $a,b,j\in\ozn{\kappa-1}$. Then the following properties hold. 
\begin{enumerate}[\rm (i)]
 \item $\fil(j)\leq j\leq\fir(j)$,
    \item $\mu_{\fir(j)}=\mu_k=\mu_{\fil(j)}$ for any integer $k\in\{\fil(j),\dots,\fir(j)\}$,
    \item $\bigcup_{\fil(j)\leq l\leq \fir(j)}\Big[\min_{k\leq l}\sA_{\inv{k}},\min_{k<l}\sA_{\inv{k}}\Big)=\Big[\min_{k\leq\fir(j)}\sA_{\inv{k}},\min_{k<\fil(j)}\sA_{\inv{k}}\Big)$,
    \item $\fil$ and $\fir$ are nondecreasing.
\end{enumerate}
\end{proposition}
\begin{proof} The statements (i) and (ii) follow directly from definitions of $\fir$, $\fil$, and arrangement~\eqref{Fi}.
The validity of the statement (iii) follows from the fact that 
$\min_{k<l}\sA_\inv{k}=\min_{k\leq l-1}\sA_\inv{k}$ for $l\in\{\fil(j),\dots,\fir(j)\},$ $l\neq 0$, and because of 
$$\min_{k\leq\fir(j)}\sA_\inv{k}\leq\min_{k<\fir(j)}\sA_\inv{k}=\min_{k\leq\fir(j)-1}\sA_\inv{k}\leq\dots\leq\min_{k<\fil(j)+1}\sA_\inv{k}=\min_{k\leq\fil(j)}\sA_\inv{k}\leq\min_{k<\fil(j)}\sA_\inv{k}$$
with the convention already stated in this paper $\min\limits_{k< 0}\sA_\inv{k}=\min \emptyset=+\infty.$
Now we show (iv). If $a\leq b$, then from arrangement~\eqref{Fi} we have $\mu_a\leq\mu_b$ and thus $$\max\{k\in\ozn{\kappa-1}:\mu_a=\mu_k\}\leq\max\{k\in\ozn{\kappa-1}:\mu_b=\mu_k\},$$ hence $\fir(a)\leq\fir(b)$.
The second part of the statement (iv) can be shown analogously. 
\qed
\end{proof}
\medskip


By means of $\fir$ and $\fil$ we can determine the greatest possible intervals corresponding to given monotone measure. Simultaneously, we show under what conditions the value of a given monotone measure is not achieved by the generalized survival function. 

\begin{proposition}\label{najgsf}
Let $\cA$ be a FCA, $\mu\in\mathbf{M}$, $\x\in[0,+\infty)^{[n]}$, $\fil$ and $\fir$ be given as in~(\ref{phii}), and $j\in\ozn{\kappa-1}$. Then 
 $$\gsf{}{\x}{\alpha}=\mu_j \text{ for any } \alpha\in\Big[\min_{k\leq\fir(j)}\sA_\inv{k},\min_{k<\fil(j)}\sA_\inv{k}\Big)$$
 with the convention $\min\limits_{k< 0}\sA_\inv{k}=\min \emptyset=+\infty$ and this interval is the greatest possible. 
\end{proposition}
\begin{proof}
According to Theorem~\ref{gsf2}, Proposition~\ref{vlastnosti_fi}(i), (ii), (iii)  the generalized survival function $\gsf{}{\x}{\alpha}$ achieves the value $\mu_j$ on interval $\Big[\min_{k\leq\fir(j)}\sA_\inv{k},\min_{k<\fil(j)}\sA_\inv{k}\Big)$. However, it is not clear that this interval is the greatest possible. By contradiction:
So let exist $\tilde{j}\in\ozn{\kappa-1}$ such that $$\min_{k\leq\fir(\tilde{j})}\sA_\inv{k}<\min_{k\leq\fir(j)}\sA_\inv{k}\quad\text{or}\quad\min_{k<\fil(j)}\sA_\inv{k}<\min_{k<\fil(\tilde{j})}\sA_\inv{k},$$
and $\gsf{}{\x}{\alpha}=\mu_j$ for $\alpha\in\Big[\min_{k\leq\fir(\tilde{j})}\sA_\inv{k},\min_{k\leq\fir(j)}\sA_\inv{k}\Big)$ or $\alpha\in\Big[\min_{k<\fil(j)}\sA_\inv{k},\min_{k<\fil(\tilde{j})}\sA_\inv{k}\Big)$.
Let us discuss these cases.
\begin{itemize}
    \item By Theorem~\ref{gsf2}(i) for $\alpha=\min_{k\leq\fir(\tilde{j})}\sA_\inv{k}$ it holds $\gsf{}{\x}{\alpha}=\mu_{\fir(\tilde{j})}=\mu_{\tilde{j}}$, where the last equality holds because of Proposition~\ref{vlastnosti_fi}(i), (ii).
    However, $\mu_{\tilde{j}}\neq\mu_j$, otherwise we get a~contradiction.
    Indeed, if $\mu_{\tilde{j}}=\mu_j$, then $\fir(j)=\fir(\tilde{j})$ and thus $\min_{k\leq\fir(\tilde{j})}\sA_\inv{k}=\min_{k\leq\fir(j)}\sA_\inv{k}$, which is in conflict with choice of $\tilde{j}$.
    \item Using the same arguments as in the previous case, for $\alpha=\min_{k<\fil(j)}\sA_\inv{k}=\min_{k\leq\fil(j)-1}\sA_\inv{k}$, we have $\gsf{}{\x}{\alpha}=\mu_{\fil(j)-1}$.
    Further, $\mu_{\fil(j)-1}<\mu_{\fil(j)}=\mu_j$, thus $\mu_{\fil(j)-1}\neq\mu_j$.
\end{itemize}
This completes the proof.
\qed
\end{proof}

\bigskip
From the above proposition, we immediately obtain the following result.

\begin{corollary}\label{nenadobuda}
Let $\cA$ be a FCA, $\mu\in\mathbf{M}$, $\x\in[0,+\infty)^{[n]}$, $\fil$ and $\fir$ be given as in~(\ref{phii}), and $j\in\ozn{\kappa-1}$.
The value $\mu_j$ is not achieved if and only if $\min_{k\leq\fir(j)}\sA_\inv{k}=\min_{k<\fil(j)}\sA_\inv{k}$.
\end{corollary}

\begin{lema}\label{nenadobuda3} 
    Let $\cA$ be a FCA, $\mu\in\mathbf{M}$, $\x\in[0,+\infty)^{[n]}$, $\fil$ and $\fir$ be given as in~(\ref{phii}). 
      \begin{enumerate}[\rm(i)]
    \item Let $j\in[\kappa-1]$. Then it holds
    \begin{enumerate}
        \item[\rm(i1)] If $\mu_j>\mu_{j-1}$, then $\fil(j)-1=\fir(j-1)$ and $\mu_j$ is achieved on $\big[\min\limits_{k\leq \fir(j)}\sA_\inv{k}, \min\limits_{k\leq \fir(j-1)}\sA_\inv{k}\big)$. Moreover, this interval is the greatest possible.
     \item[\rm(i2)] If $\mu_j=\mu_{j-1}$, then $\fir(j)=\fir(j-1)$.  
  \end{enumerate}
         \item Let $j\in\ozn{\kappa-1}$. Then it holds
\begin{enumerate}
        \item[\rm(ii1)] If $\mu_j<\mu_{j+1}$, then $\fir(j)+1=\fil(j+1)$ and  
        $\mu_j$ is achieved on $\Big[\min_{k<\fil(j+1)}\sA_\inv{k},\min_{k<\fil(j)}\sA_\inv{k}\Big)$, with $\min\limits_{k< \fil(\kappa)}\sA_\inv{k}=0$ by convention. Moreover, this interval is the greatest possible.
\item[\rm(ii2)]  If $\mu_j=\mu_{j+1}$, then $\fil(j)=\fil(j+1)$.
 \end{enumerate}
     \end{enumerate}
\end{lema}

\begin{proof} Let us prove part (i). 
\begin{enumerate}
    \item[(i1)] If $\mu_j>\mu_{j-1}$, then directly from definitions of $\fil$ and $\fir$ we have $\fil(j)=j$ and $\fir(j-1)=j-1$, thus $\fir(j-1)=\fil(j)-1$.
    Further, according to Proposition~\ref{najgsf}, we have
$$\gsf{}{\x}{\alpha}=\mu_j$$
for each $\alpha\in\big[\min\limits_{k\leq \fir(j)}\sA_\inv{k}, \min\limits_{k\leq \fil(j)-1}\sA_\inv{k}\big)$ and this interval is the greatest possible. 
Using the above equality
we have $$\big[\min\limits_{k\leq \fir(j)}\sA_\inv{k}, \min\limits_{k\leq \fil(j)-1}\sA_\inv{k}\big)=\big[\min\limits_{k\leq \fir(j)}\sA_\inv{k}, \min\limits_{k\leq \fir(j-1)}\sA_\inv{k}\big).$$ 
\item[(i2)] It holds trivially.
\end{enumerate}
  Part (ii) can be proved analogously.
\qed
\end{proof}

Let us notice that Proposition~\ref{najgsf} and its corollaries are crucial to state the sufficient conditions for indistinguishability of generalized Choquet integral equivalent pairs, see Section~\ref{Choquet}. Using the above results, one can obtain the improvement of formula~(\ref{vyjgsf2}).

\begin{proposition}\label{skrateniegsf2}Let $\cA$ be a FCA, $\mu\in\mathbf{M}$, $\x\in[0,+\infty)^{[n]}$, and $\varphi$ be given as in~(\ref{phii}). Then 
\begin{align}\label{skrateniegsf2f}
\gsf{}{\x}{\alpha}&=\sum_{j=1}^{\kappa-1}\mu_j\bin_{\big[\min\limits_{k\leq \fir(j)}\sA_\inv{k},\min\limits_{k\leq \fir(j-1)}\sA_\inv{k}\big)}(\alpha)=\sum_{j=1}^{\kappa-1}\mu_j\bin_{\big[\min\limits_{k< \fil(j+1)}\sA_\inv{k},\min\limits_{k< \fil(j)}\sA_\inv{k}\big)}(\alpha)\
\end{align} 
for any $\alpha\in[0,\infty)$, with the conventions $\fil(\kappa):=\kappa$ (thus $\min\limits_{k< \fil(\kappa)}\sA_\inv{k}=0$), $\min\limits_{k< 0}\sA_\inv{k}=\min \emptyset=+\infty$.
\end{proposition}
\begin{proof} Let us consider an arbitrary (fixed) $j\in[\kappa-1]$. If $\mu_j>\mu_{j-1}$, then according to Lemma~\ref{nenadobuda3} (i1) $\mu_j$ is achieved on $\Big[\min_{k\leq\fir(j)}\sA_\inv{k},\min_{k\leq\fir(j-1)}\sA_\inv{k}\Big)$. Moreover, this interval is the greatest possible.  
If $\mu_j=\mu_{j-1}$, then $\fir(j)=\fir(j-1)$, see Lemma~\ref{nenadobuda3}(i2), thus $\Big[\min_{k\leq\fir(j)}\sA_\inv{k},\min_{k\leq\fir(j-1)}\sA_\inv{k}\Big)=\emptyset$. This demonstrates that each value of generalized survival function is included just once in the sum and the first formula is right. 
The second formula can be proved analogously. 
\qed
\end{proof}

If we use the mapping $\bj$ defined by~(\ref{jj}), 
then formulas in~(\ref{skrateniegsf2f}) can be rewritten similarly as in Corollary~\ref{application_j}, i.e., $$\gsf{}{\x}{\alpha}=\sum_{j=1}^{\kappa-1}\mu_j\bin_{\big[\sA_{\bj(\fir(j))},\sA_{\bj(\fir(j-1
))}\big)}(\alpha)=\sum_{j=1}^{\kappa-1}\mu_j\bin_{\big[\sA_{\bj(\fil(j+1)-1)},\sA_{\bj(\fil(j)-1)}\big)}(\alpha).
$$
The use of the approach described in the previous proposition is shown in the following example.

\begin{example}\label{example_j}
\rm Let us consider the same inputs as in Example~\ref{graphical_representation}. 
Let us use the last formula given in~\eqref{skrateniegsf2f} to calculate the generalized survival function.
\begin{itemize*}
\item For $j=1$ we get $\mu_1\cdot\bin_{[\sA_{\bj(0)},\sA_{\bj(0)})}=0.5\cdot\bin_{\emptyset}$.
\item For $j=2$ we get $\mu_2\cdot\bin_{[\sA_{\bj(0)},\sA_{\bj(0)})}=0.5\cdot\bin_{\emptyset}$.
\item For $j=3$ we get $\mu_3\cdot\bin_{[\sA_{\bj(3)},\sA_{\bj(0)})}=\mu_3\cdot\bin_{[\sA_{1},\sA_{5})}=0.5\cdot\bin_{[1,6)}$.
\item For $j=4$ we get $\mu_4\cdot\bin_{[\sA_{\bj(4)},\sA_{\bj(3)})}=\mu_4\cdot\bin_{[\sA_{1},\sA_{1})}=0.7\cdot\bin_{\emptyset}$.
\item For $j=5$ we get $\mu_5\cdot\bin_{[\sA_{\bj(5)},\sA_{\bj(4)})}=\mu_5\cdot\bin_{[\sA_{0},\sA_{1})}=1\cdot\bin_{[0,1)}$.
\end{itemize*}
Therefore, the generalized survival function has the form
\begin{align*}
    \gsf{\mathrm{sum}}{\x}{\alpha}=\bin_{[0,1)}(\alpha)+0.5\cdot\bin_{[1,6)}(\alpha)\text{,}
\end{align*}
$\alpha\in[0,+\infty)$, compare with Example~\ref{priklad2}.
\end{example}



As in the whole paper let us introduce the greatest possible intervals on which a value of monotone measure is achieved using bijection $E_*\colon \ozn{\kappa-1}\to{\cE}$, see~\eqref{EiAi}. Since the proofs of the following propositions are analogous, we omit them. Let $i\in\ozn{\kappa-1}$, let us define
\begin{align}\label{psii}
\begin{split}
\psil(i)&:=\min\{l\in\ozn{\kappa-1}:\min_{k\leq l}\mu_{(k)}=\min_{k\leq i}\mu_{(k)}\}\text{,}\\\text{and}\quad\psir(i)&:=\max\{l\in\ozn{\kappa-1}:\min_{k\leq l}\mu_{(k)}=\min_{k\leq i}\mu_{(k)}\}\text{.}
\end{split}
\end{align}

\begin{proposition}\label{vlastnosti_psi} Let $\cA$ be a FCA, $\mu\in\mathbf{M}$, $\x\in[0,\infty)^{[n]}$, $\psir$, $\psil$ be defined as in~\eqref{psii} and $a,b,i\in\ozn{\kappa-1}$. Then the following properties hold. 
\begin{enumerate}[\rm (i)]
\item $\psil(i)=\min\{l\in\ozn{\kappa-1}:\mu_{(l)}=\min_{k\leq i}\mu_{(k)}\}\text{.}$
 \item $\psil(0)=0$,  $\psil(i)\leq i\leq\psir(i)$.
    \item For any integer $u\in\{\psil(i),\dots,\psir(i)\}$ it holds $\min_{k\leq u}\mu_{(k)}=\min_{k\leq i}\mu_{(k)}=\mu_{(\psil(i))}=\mu_{\bi(i)}$.
     \item $\min_{k\leq a}\mu_{(k)}=\min_{k\leq b}\mu_{(k)}$ if only if $\psil(a)=\psil(b)$ if only if $\psir(a)=\psir(b)$.
        \item $\min_{k\leq a}\mu_{(k)}>\min_{k\leq b}\mu_{(k)}$ if only if $\psil(a)<\psil(b)$ if only if $\psir(a)<\psir(b)$.
        \item $\psil$ and $\psir$ are nondecreasing.
       \end{enumerate}
\end{proposition}
\begin{proof} See Appendix.\qed
\end{proof}



\begin{remark}
In general, we cannot determine the order of the values $\mu_{(i)}$, $i\in\ozn{\kappa-1}$.
    However, it holds
    \begin{align*}
        &\mu_{(\psil(0))}\geq\mu_{(\psil(1))}\geq\dots\geq\mu_{(\psil(\kappa-1))}\text{,}\\
        \text{and}\quad& \mu_{(\psir(0))}\geq\mu_{(\psir(1))}\geq\dots\geq\mu_{(\psir(\kappa-1))}\text{.}
    \end{align*}
    This property follows directly from the definitions of $\psil$ and $\psir$.
    Indeed, we have already shown that for $a\leq b$ we get $\psil(a)\leq\psil(b)$.
    In case that $\psil(a)=\psil(b)$, the result is clear. If $\psil(a)<\psil(b)$, then the result follows from Proposition~\ref{vlastnosti_psi} (vi). Similarly for $\psir$.
\end{remark}
\medskip

In Theorem~\ref{zjednodusenie_def} and Corollary~\ref{application_i}, we have described the value of the monotone measure acquired on a given interval. 
However, the same value can be acquired on several intervals. 
Using $\psir$ and $\psil$, we can determine the greatest interval on which the value of the monotone measure is achieved.

\begin{proposition}\label{naj_gsf_cez_psi}
Let $\cA$ be a FCA, $\mu\in\mathbf{M}$, $\x\in[0,+\infty)^{[n]}$, $\psil$ and $\psir$ be given as in~(\ref{psii}), and $i\in\ozn{\kappa-1}$. Then 
$$\gsf{}{\x}{\alpha}=\min_{k\leq i}\mu_{(k)}
\text{ for any } \alpha\in\big[\sA_{\psil(i)},\sA_{\psir(i)+1}\big)$$ and this interval is the greatest possible. 
\end{proposition}
\begin{proof} See Appendix.\qed
\end{proof}

\begin{lema}\label{lema_ku_psi} 
Let $\cA$ be a FCA, $\mu\in\mathbf{M}$, $\x\in[0,+\infty)^{[n]}$, $\psil$ and $\psir$ be given as in~(\ref{psii}). 
\begin{enumerate}[\rm(i)]
    \item Let $i\in{\ozn{\kappa-2}}$ . Then it holds
        \begin{enumerate}
            \item[\rm(i1)] If $\min_{k\leq i}\mu_{(k)}>\min_{k\leq i+1}\mu_{(k)}$, then $\psir(i)+1=\psil(i+1)$ and $\min_{k\leq i}\mu_{(k)}
            $ is achieved on $\big[\sA_{\psil(i)}, \sA_{\psil(i+1)}\big)$. Moreover, this interval is the greatest possible.
            \item[\rm(i2)] If $\min_{k\leq i}\mu_{(k)}=\min_{k\leq i+1}\mu_{(k)}$, then $\psil(i)=\psil(i+1)$.  
        \end{enumerate}
    \item Let $i\in{{[\kappa-1]}}$. Then it holds
        \begin{enumerate}
            \item[\rm(ii1)] If $\min_{k\leq i-1}\mu_{(k)}<\min_{k\leq i}\mu_{(k)}$, then $\psil(i)=\psir(i-1)+1$ and  
            $\min_{k\leq i}\mu_{(k)}
            $ is achieved on $\Big[\sA_{\psir(i-1)+1},\sA_{\psir(i)+1}\Big)$.
            Moreover, this interval is the greatest possible.
            \item[\rm(ii2)]  If $\min_{k\leq i-1}\mu_{(k)}=\min_{k\leq i}\mu_{(k)}$, then $\psir(i-1)=\psir(i)$.
        \end{enumerate}
\end{enumerate}
\end{lema}

\begin{proposition}\label{skratenie_cez_psi}
Let $\cA$ be a FCA, $\mu\in\mathbf{M}$, $\x\in[0,+\infty)^{[n]}$, $\psil$ and $\psir$ be given as in~(\ref{psii}), and $i\in\ozn{\kappa-1}$. Then 
the formula of generalized survival function is as follows 
\begin{align}\label{skrateniegsf2fA}
\gsf{}{\x}{\alpha}=\sum_{i=0}^{\kappa-2}\mu_{(\psil(i))}\bin_{\big[\sA_{\psil(i)},\sA_{\psil(i+1)})}(\alpha)=\sum_{i={0}}^{\kappa-2}{\mu_{(\psil(i))}}\bin_{\big[\sA_{\psir(i-1)+1},\sA_{\psir(i)+1})}(\alpha)
\end{align} 
for any $\alpha\in[0,+\infty)$ with the convention $\psir(-1)=0$.
\end{proposition}
\begin{proof}
    See Appendix.
    \qed
\end{proof}

If we use the mapping $\bi$ defined by~(\ref{ii}), then formulas in~(\ref{skrateniegsf2fA}) can be rewritten similarly as in Corollary~\ref{application_i}, i.e.,
$$\gsf{}{\x}{\alpha}=\sum_{i=0}^{\kappa-2}\mu_{\bi(i)}\bin_{\big[\sA_{\psil(i)},\sA_{\psil(i+1)})}(\alpha)=\sum_{i={0}}^{\kappa-2}{\mu_{\bi(i)}}\bin_{\big[\sA_{\psir(i-1)+1},\sA_{\psir(i)+1})}(\alpha)\text{.}$$

As we have already mentioned, searching for optimal intervals will be helpful for  studying  the indistinguishability of generalized survival functions. In the following, we state sufficient and necessary conditions under which the generalized survival functions coincide. This is applicable to decision-making problems. In fact, if the generalized survival functions of two alternatives (e.g.\ two offers of accommodation) are the same, then their overall score will be the same. Both alternatives will be in the same place in the ranking.

\begin{definition}The triples $(\mu, \cA, \x)$ and $(\mu^\prime, \cA^\prime, \x^\prime)$, where $\mu,\mu^\prime$ are monotone measures, $\cA,\cA^\prime$ are FCA and $\x,\x^\prime$ are vectors, are called \textit{integral equivalent}, if $$\gsf{}{\x}{\alpha}=\gsfp{\prime}{\x^\prime}{\alpha}.$$\end{definition}
\smallskip

\begin{proposition}\label{prop_nerozlisitelnost}
Let $\cA$ and $\cA^\prime$ be a FCA, $\mu,\mu^{\prime}\in\mathbf{M}$, $\x,\x^\prime\in[0,+\infty)^{[n]}$ and $\fir$ be given by~(\ref{phii}).
Then the following assertions are equivalent:
\begin{enumerate}[\rm (i)]
\item $\gsf{}{\x}{\alpha}=\mu^\prime_{\cA^{\prime}}(\x^{\prime},\alpha)$ for any $\alpha\in[0,\infty)$;
\item for each $j\in\ozn{\kappa-1}$ with
$\min_{k\leq\fir(j)}\sA(\x|F_k^c)<\min_{k<\fil(j)}\sA(\x|F_k^c)$ 
there exists $j^{\prime}\in\ozn{\kappa^\prime-1}$ such that $\mu_{j}=\mu^\prime_{j^{\prime}}$, $\min_{k\leq\fir(j)}\sA(\x|F_k^c)=\min_{k\leq\fir(j^\prime)}\sA^\prime(\x^\prime|F_k^c)$  and $\min_{k<\fil(j)}\sA_\inv{k}=\min_{k<\fil(j^\prime)}\sA^\prime_\inv{k}$.

\end{enumerate}
\end{proposition}
\begin{proof}
It follows from Proposition~\ref{najgsf}.
\qed   
\end{proof}
\smallskip

\begin{remark}
   Following Proposition~\ref{naj_gsf_cez_psi}, condition (ii) in the previous proposition can be equivalently formulated as follows:
for each $i\in\ozn{\kappa-1}$ with
$\sA_{\psil(i)}<\sA_{\psir(i)+1}$ 
there exists $i^{\prime}\in[\kappa^\prime-1]$ such that $\mu_{\psil(i)}=\mu^\prime_{\psil(i^{\prime})}$, $\sA(\x|E_{\psil(i)})=\sA^{\prime}(\x^{\prime}|E_{\psil(i^{\prime})})$ and $\sA(\x|E_{\psir(i)+1})=\sA^{\prime}(\x^{\prime}|E_{\psir(i^{\prime})+1})$.

\end{remark}

Fixing a collection $\cE$ and a monotone measure $\mu$ we obtain the following sufficient and necessary condition for integral equivalence of triples  $(\mu, \sA, \x)$ and $(\mu, \sA^\prime, \x^\prime)$. 

\begin{corollary}
Let $\bA=\{\aA{\cdot}: E\in\collectionGSF\}$ and $\cA^\prime=\{\aAp{\cdot}: E\in\collectionGSF\}$ be a FCA, $\mu\in\mathbf{M}$, $\x,\x^\prime\in[0,+\infty)^{[n]}$, and $\fil,\fir$ and $\psil,\psir$ be given by~\eqref{phii}, \eqref{psii}, respectively.
Then the following assertions are equivalent:
\begin{enumerate}[\rm (i)]
\item $\gsf{}{\x}{\alpha}=\gsf{\prime}{\x^\prime}{\alpha}$ for any $\alpha\in[0,+\infty)$;
\item $\min_{k\leq\fir(j)}\sA(\x|F_k^c)=\min_{k\leq\fir(j)}\sA^\prime(\x^\prime|F_k^c)$
for any $j\in\ozn{\kappa-1}$\\  (or equivalently, $\sA(\x|E_{\psil(i)})=\sA^{\prime}(\x^{\prime}|E_{\psil(i)})$ for each $i\in\ozn{\kappa-1}$). 
\end{enumerate}
\end{corollary}



 \begin{remark} In~\cite{ChenMesiarLiStupnanova2017}, the authors derived the necessary 
 and sufficient condition of equality $\mu(\{\x\geq\alpha\})=\mu(\{\y\geq\alpha\})$ with $\mu$ being a possibility and necessity measure, respectively. Our result includes equality $\mu(\{\x>\alpha\})=\mu(\{\y>\alpha\})$ with $\mu$ being an arbitrary monotone measure.
\end{remark}

\begin{example}
Let us consider the collection $\collectionGSF=\{\emptyset,\{1\},\{1,2,3\}\}$, the families of conditional aggregation operators $\cA^\mathrm{max}=\{\aAi[E]{\cdot}{\mathrm{max}}:E\in\collectionGSF\}$ and $\cA^\mathrm{sum}=\{\aAi[E]{\cdot}{\mathrm{sum}}:E\in\collectionGSF\}$, the vectors $\x=(2,5,9)$ and $\x^\prime=(2,3,4)$, and the monotone measure $\mu\in\mathbf{M}$ with corresponding values in table.
\begin{table}[H]
\renewcommand*{\arraystretch}{1.2}
\begin{center}
\begin{tabular}{|c|M|c|c|}
\hline
$j$ & 0 & 1 & 2 \\ \hline
$F_j$ & $\emptyset$ & $\{2,3\}$ & $\{1,2,3\}$ \\ \hline
$\mu_j$ & $0$ & $0.5$ & $1$ \\ \hline
$\sA^{\textrm{max}}(\x|F_j^c)=\sA^{\mathrm{sum}}(\x^\prime|F_j^c)$ & $9$  & $2$ & $0$ \\\hline

\end{tabular}
\end{center}
\end{table}
\noindent It is easy to see that for any $j\in\ozn{2}$ we have $\fir(j)=j$  and  $\min_{k\leq j}\sA^{\mathrm{max}}(\x|F_k^c)=\min_{k\leq j}\sA^{\mathrm{sum}}(\x^\prime|F_k^c)$. Then according to previous proposition one can expect the equality of corresponding generalized survival functions. And, it is:
$$\gsf{\mathrm{max}}{\x}{\alpha}=\gsf{\mathrm{sum}}{\x^\prime}{\alpha}=1\cdot\mathbf{1}_{[0,2)}+0.5\cdot\mathbf{1}_{[2,9)}$$
for any $\alpha\in[0,+\infty)$.
\end{example}

The following example demonstrates a standard situation in decision-making processes. Different alternatives can be evaluated with the same score, thus they are incomparable. It is not possible to decide which one is better.

\begin{example}
Let us consider the collection $\collectionGSF=\{\emptyset,\{2\},\{3\},\{1,2\},\{1,3\},\{1,2,3\}\}$, the family of conditional aggregation operators $\cA^\mathrm{sum}=\{\aAi[E]{\cdot}{\mathrm{sum}}:E\in\collectionGSF\}$, vectors $\x=(1,3,5)$ and $\x^\prime=(2,4,3)$, and the monotone measure $\mu\in\mathbf{M}$ with corresponding values in table.
\begin{table}[H]
\renewcommand*{\arraystretch}{1.2}
\begin{center}
\begin{tabular}{|c|M|c|c|c|c|c|}
\hline
$j$ & 0 & 1 & 2 & 3 & 4 & 5 \\ \hline
$F_j$ & $\emptyset$ & $\{1,3\}$ & $\{1,2\}$ & $\{3\}$ & $\{2\}$ & $\{1,2,3\}$ \\ \hline
$\mu_j$ & $0$ & $0.5$ & $0.5$ & $0.5$ & $0.5$ & $1$ \\ \hline
$\sA^{\mathrm{sum}}(\x|F_k^c)$ & $9$ & $3$ & $5$ & $4$ & $6$ & $0$ \\ \hline
$\sA^{\mathrm{sum}}(\x^\prime|F_k^c)$ & $9$ & $4$ & $3$ & $6$ & $5$ & $0$ \\ \hline
\end{tabular}
\end{center}
\end{table}
\noindent It is easy to see that $\fir(0)=0, \fir(1)=\fir(2)=\fir(3)=\fir(4)=4, \fir(5)=5$  and  
\begin{center}
$\min_{k\leq 0}\sA^{\mathrm{max}}(\x|F_k^c)=9=\min_{k\leq 0}\sA^{\mathrm{sum}}(\x^\prime|F_k^c)$,\\ $\min_{k\leq 4}\sA^{\mathrm{max}}(\x|F_k^c)=3=\min_{k\leq 4}\sA^{\mathrm{sum}}(\x^\prime|F_k^c)$,\\ $\min_{k\leq 5}\sA^{\mathrm{max}}(\x|F_k^c)=0=\min_{k\leq 5}\sA^{\mathrm{sum}}(\x^\prime|F_k^c)$. 
\end{center} Then according to previous proposition one can expect the equality of corresponding generalized survival functions. And, it is:
$$\gsf{\mathrm{sum}}{\x}{\alpha}=\gsf{\mathrm{sum}}{\x^\prime}{\alpha}=1\cdot\mathbf{1}_{[0,3)}+0.5\cdot\mathbf{1}_{[3,9)}$$
for any $\alpha\in[0,+\infty)$.
\end{example}

Immediately, we get the sufficient condition for equality of generalized Choquet integrals defined w.r.t.\ different FCA and vectors.
\begin{corollary}\label{rovnostCh}Let $\cA$ and $\cA^\prime$ be a FCA, $\mu\colon\hat{\cE}\to[0,+\infty)$, $\x,\x^\prime\in[0,+\infty)^{[n]}$ and $\varphi$ be given by~(\ref{phii}).
If $\min_{k\leq\varphi(i)}\aA[F_k^c]{\x}=\min_{k\leq\varphi(i)}\aAi[F_k^c]{\x^\prime}{\prime}$ for any $\varphi(i)\in[\kappa]$, then $$\mathrm{C}_{\cA}(\x,\mu)=\mathrm{C}_{\cA^\prime}(\x^\prime,\mu).$$
\end{corollary}

\section*{Conclusion}


In the paper, we dealt with the concept of the generalized survival function and the generalized Choquet integral related to it. Considering their applications, see Section 2, we were mainly interested in their computational formulas on discrete space.
The derivation of these formulas required the introduction of new notations whose idea and practical meaning can be visually interpreted, as we described  in Sections 3 and 4. Interesting results are Proposition~\ref{zjednodusenie_def}, Proposition~\ref{gsf2}, Corollary~\ref{application_i} and Corollary~\ref{application_j}. 
In Section 4 we also pointed out the direct and efficient construction of the graph of the generalized survival function.

Motivated by applications we solved the indistinguishability of generalized survival functions. This question is interesting, especially in decision-making processes, because the indistinguishability means incomparability of two inputs (alternatives).
The interesting results given in Proposition~\ref{skrateniegsf2}, Proposition~\ref{skratenie_cez_psi} and Proposition~\ref{prop_nerozlisitelnost} are related to this.

In this paper, we also mentioned the formulas for calculating the generalized Choquet integral with respect to special types of monotone measure, and we solved the introductory problems serving as a motivation to study mentioned concepts.

\section*{Acknowledgments}
The work was supported the grants APVV-21-0468, VEGA 1/0657/22, and grant scheme VVGS-PF-2022-2143.


\section*{Appendix}
\textbf{Proof of Corollary~\ref{specimiery}}
\begin{enumerate}[(i)]
\item Clearly, in accordance with~(\ref{Fi}) we have $0=\mu_0=\mu(\emptyset)$ and according to Theorem~\ref{gsf2} it is achieved on $\Big[\min_{k\leq 0}\sA_{\inv{k}},+\infty\Big)=[\sA_{\inv{0}},+\infty)=[\aA[{[n]}]{\x},+\infty)$.
Then $1=\mu_1=\dots=\mu_{\kappa-1}$ is achieved on $[0,\aA[{[n]}]{\x})$.

\item From Theorem~\ref{gsf2}, the value $1=\mu_{\kappa-1}$ is achieved on $\Big[0, \min_{k<\kappa-1}\sA_{\inv{k}}\Big)=\Big[0,\min_{E\neq \emptyset}\aA[E]{\x}\Big).$




\item 
Let us consider the bijection $F_{*}$ given in~\eqref{Fi} such that $|F_{k}|<|F_{l}|$ implies $k<l$ and for any $i\in[n]_0$ let us set \begin{align*}\omega^{*}(i)&=\max\{j\in[\kappa-1]_0:|F_j|=i\},\,\,
\omega_{*}(i)=\min\{j\in[\kappa-1]_0:|F_j|=i\}. \end{align*}
It is easy to see that $\omega_{*}(i)-1=\omega^{*}(i-1)$ with the convention $\omega^{*}(-1)=-1$. Further, because of monotonicity of $\mu$ and FCA for any $i\in[n]_0$ we have 
\begin{align*}
 \min_{k\leq\omega^{*}(i)}\sA_{\inv{k}}=  \min_{k\leq\omega^{*}(i)}\sA(\x|F_k^c)=\min_{k\in[\omega_{*}(i),\omega^{*}(i)]}\sA(\x|F_k^c)=\min_{|E|=n-i}\sA(\x|E).
\end{align*}
Indeed, for any $k<\omega_{*}(i)$ since $\collectionGSF=2^{[n]}$ there exists $\tilde{k}\in[\omega_{*}(i), \omega^{*}(i)]$ such that $F_{\tilde{k}}\supset F_{k}$.
Then $\sA(\x|F_{\tilde{k}}^c)\leq \sA(\x|F_{k}^c)$.
According to Theorem~\ref{gsf2} we have that $\mu^i$ is achieved on 
\begin{align*}\bigcup_{j\in\{\omega_{*}(i),\dots,\omega^{*}(i)\}}\Big[\min_{k\leq j}\sA_{\inv{k}},\min_{k< j}\sA_{\inv{k}}\Big)&=\Big[\min_{k\leq \omega^{*}(i)}\sA_{\inv{k}},\min_{k< \omega_{*}(i)}\sA_{\inv{k}}\Big)\\&=\Big[\min_{k\leq \omega^{*}(i)}\sA_{\inv{k}},\min_{k\leq \omega^{*}(i-1)}\sA_{\inv{k}}\Big)\\&=\Big[\min_{E=n-i}\sA_{\inv{k}},\min_{E=n-i+1}\sA_{\inv{k}}\Big).
\end{align*}
\item 
Let us consider the bijection $F_{*}$ given in~\eqref{Fi} such that $F_0=\emptyset$ and for any $i\in\{2,3,\dots,n\}$
\begin{center}
$\{\sigma(i-1)\}\subseteq F_{k}\subseteq G^c_{\sigma(i)}$, 
$\{\sigma(i)\}\subseteq F_{l}\subseteq G^c_{\sigma(i+1)}$, implies $k<l$.  
\end{center}
According to Theorem~\ref{gsf2} it is clear that $\pi(\sigma(0))$ is achieved on $\Big[\sA(\x|G_{\sigma(1)}),+\infty\Big)$. Further, for any $i\in[n]$ let us set \begin{align*}\tau^{*}(\sigma(i))&=\max\{j\in[\kappa-1]_0:\Pi(F_j)=\pi(\sigma(i))\,\,\text{and}\,\,\{\sigma(i)\}\subseteq F_j\subseteq G^c_{\sigma(i+1)}\},\,\,\\
\tau_{*}(\sigma(i))&=\min\{j\in[\kappa-1]_0:\Pi(F_j)=\pi(\sigma(i))\,\,\text{and}\,\,\{\sigma(i)\}\subseteq F_j\subseteq G^c_{\sigma(i+1)}\}. \end{align*}
It is easy to see that $\tau_{*}(\sigma(i))-1=\tau^{*}(\sigma(i-1))$ with the convention $\tau^{*}(0)=0$. Further, because of monotonicity of $\mu$ and FCA for any $i\in[n]_0$ we have 
\begin{align*}
 \min_{k\leq\tau^{*}(\sigma(i))}\sA_{\inv{k}}=  \min_{k\leq\tau^{*}(\sigma(i))}\sA(\x|F_k^c)=\min_{k\in\{\tau_{*}(\sigma(i)),\dots,\tau^{*}(\sigma(i))\}}\sA(\x|F_k^c)=\sA(\x|G_{\sigma(i+1)}).
\end{align*}

 Indeed, for any $k<\tau_{*}(\sigma(i))$ we have $F_{k}\subseteq G^c_{\sigma(i+1)}$, where $G^c_{\sigma(i+1)}=F_{\tilde{k}}$ for some $\tilde{k}\in[\tau_{*}(\sigma(i)),\tau^{*}(\sigma(i))]$, thus explaining the second equality. Further, for any $k\in[\tau_{*}(\sigma(i)),\tau^{*}(\sigma(i))]$ $$F_{k}\subseteq G^c_{\sigma(i+1)}$$ therefore $\sA(\x|G_{\sigma(i+1)})\leq \sA(\x|F_{k}^c)$.
According to Theorem~\ref{gsf2} we have that $\pi(\sigma(i))$ is achieved on 
\begin{align*}\bigcup_{j\in\{\tau_{*}(\sigma(i)),\dots,\tau^{*}(\sigma(i))\}}\Big[\min_{k\leq j}\sA_{\inv{k}},\min_{k< j}\sA_{\inv{k}}\Big)&=\Big[\min_{k\leq \tau^{*}(\sigma(i))}\sA_{\inv{k}},\min_{k< \tau_{*}(\sigma(i))}\sA_{\inv{k}}\Big)\\&=\Big[\min_{k\leq \tau^{*}(\sigma(i))}\sA_{\inv{k}},\min_{k\leq \tau^{*}(\sigma(i-1))}\sA_{\inv{k}}\Big)\\&=\Big[\sA(\x|G_{\sigma(i+1)}),\sA(\x|G_{\sigma(i)})\Big).
\end{align*}
\item It is clear that $N([n])=1-\Pi(\emptyset)=1$ and for each $\{\sigma(i)\}\subseteq F^c\subseteq G_{\sigma(i+1)}^c$  $$N(F)=1-\Pi(F^c)=1-\pi(\sigma(i)),$$
with $i\in[n]$ and $$0=1-\pi(\sigma(n))\leq \dots \leq 1-\pi(\sigma(i))\leq\dots\leq 1-\pi(\sigma(0))=1.$$ Let us consider the bijection $F_{*}$ given in~\eqref{Fi}
such that $F_{\kappa-1}=\emptyset$ and for any $i\in\{2,3,\dots,n\}$
\begin{center}
$\{\sigma(i-1)\}\subseteq F_{k}\subseteq G^c_{\sigma(i)}$, 
$\{\sigma(i)\}\subseteq F_{l}\subseteq G^c_{\sigma(i+1)}$, implies $k>l$.  
\end{center}
According to Theorem~\ref{gsf2} it is clear that $0=1-\pi(\sigma(n))$ is achieved on $\Big[\sA(\x|\{\sigma(n)\}),+\infty\Big)$. Further, for any $i\in[n-1]$ let us set \begin{align*}\rho^{*}(\sigma(i))&=\max\{j\in[\kappa-1]_0:N(F_j)=1-\pi(\sigma(i))\,\,\text{and}\,\,\{\sigma(i)\}\subseteq F_j^c\subseteq G_{\sigma(i+1)}^c\}\\
\rho_{*}(\sigma(i))&=\min\{j\in[\kappa-1]_0:N(F_j)=1-\pi(\sigma(i))\,\,\text{and}\,\,\{\sigma(i)\}\subseteq F_j^c\subseteq G_{\sigma(i+1)}^c\}.
\end{align*}
It is easy to see that $\rho^{*}(\sigma(i))\geq\rho_{*}(\sigma(i))>0$ and $\rho_{*}(\sigma(i))-1=\rho^{*}(\sigma(i+1))$. Further, because of monotonicity of $\mu$ and FCA for any $i\in[n-1]_0$ we have 
\begin{align*}
 \min_{k\leq\rho^{*}(\sigma(i))}\sA_{\inv{k}}=  \min_{k\leq\rho^{*}(\sigma(i))}\sA(\x|F_k^c)=\min_{k\geq i}\sA(\x|\{\sigma(k)\}).
\end{align*}
Indeed, for any $k\leq \rho^{*}(\sigma(i))$ there exists $\tilde{k}\geq i$ such that $N(F_k)= 1-\pi(\sigma(\tilde{k}))\leq 1-\pi(\sigma(i))$. Then $F_k^c\supseteq\{\sigma(\tilde{k})\}$. 
According to Theorem~\ref{gsf2} we have that $1-\pi(\sigma(i))$ is achieved on 
\begin{align*}&\bigcup_{j\in\{\rho_{*}(\sigma(i)),\dots,\rho^{*}(\sigma(i))\}}\Big[\min_{k\leq j}\sA_{\inv{k}},\min_{k< j}\sA_{\inv{k}}\Big)=\Big[\min_{k\leq \rho^{*}(\sigma(i))}\sA_{\inv{k}},\min_{k< \rho_{*}(\sigma(i))}\sA_{\inv{k}}\Big)\\&=\Big[\min_{k\leq \rho^{*}(\sigma(i))}\sA_{\inv{k}},\min_{k\leq \rho^{*}(\sigma(i+1))}\sA_{\inv{k}}\Big)=\Big[\min_{k\geq i}\sA(\x|\{\sigma(k)\}),\min_{k\geq i+1}\sA(\x|\{\sigma(k)\}))\Big).
\end{align*}

\end{enumerate}





\bigskip

\noindent\textbf{Proof of Proposition~\ref{porovnanie}} Let us consider an arbitrary, fixed $i\in\ozn{\kappa-1}$ such that $\sA_i\neq \sA_{i+1}$ (otherwise it is trivial). Let us denote $$j=\min\{l: \sA_{\inv{l}}\leq \sA_i\text{\,\,and\,\,} \min_{k\leq i}\mu_{(k)}=\mu_l\}.$$ 
The above-mentioned set is nonempty. Indeed, for each $l$ such that $\min_{k\leq i}\mu_{(k)}=\mu_l$ w.r.t.\ denotations from the beginning of this section there exists $i_l\leq i$ such that $F_{l}=E_{i_l}^c$ and $\mu_l=\mu_{(i_l)}$. Further, it is clear that $\sA_{\inv{l}}= \sA_{i_l}\leq \sA_{i}$.

From the definition of $j$ we immediately have $\min_{k\leq j}\sA_{\inv{k}}\leq\sA_{\inv{j}}\leq\sA_i$. Moreover, it also holds that $\min_{k< j}\sA_{\inv{k}}\geq \sA_{i+1}$. By contradiction: Let $\min_{k< j}\sA_{\inv{k}}< \sA_{i+1}$. Then there exists $k^*<j$ such that $\sA_{\inv{k^\ast}}< \sA_{i+1}$. Because of~\eqref{Ei} we get $\sA_{\inv{k^\ast}}\leq \sA_{i}$. Further, $\mu_{k^*}<\mu_j$: From the fact that $k^*<j$ we have $\mu_{k^*}\leq\mu_j$, however, the equality can not happen because of the definition of $j$. Further for $\alpha\in[\sA_i,\sA_{i+1})$  because of formula~\eqref{vyjgsf1} and because of the definition of $j$ we have 
$\gsf{}{\x}{\alpha}=\mu_{j}$. However, since $\sA_{\inv{k^\ast}}\leq \sA_{i}$ for $\alpha\in[\sA_i,\sA_{i+1})$ we get
$$\gsf{}{\x}{\alpha}=\min\left\{\mu_j: \sA_{\inv{j}}\leq\alpha\right\}\leq\mu_{k^*}<\mu_{j}.$$
This is a contradiction.\qed

\noindent\textbf{Auxiliary pseudocodes used in Algorithms~\ref{alg_GSF_from_i} and Algorithms~\ref{alg_GSF_from_j}.}

The first pseudocode, namely method \textit{find-sets-$E_i$-and-$F_i$}, describe the calculation of permutation listed in~(\ref{Ei}) and~(\ref{Fi}).
Second and third algorithms describe determination of values of $(\cdot)$, $\inv{\cdot}$.

\begin{center}
\SetKwRepeat{Do}{do}{while}
\begin{algorithm}[H]
\Fn{find-sets-$E_i$-and-$F_i$$(\collectionGSF, \mu, \x, \cA)$}{
$i:=0$, $\collectionGSF^\prime:=\collectionGSF$\;
	\Do{$\collectionGSF\neq\emptyset$}{
      $E_i\colon E_i\in\collectionGSF\text{ and }\aA[E_i]{\x}=\min\{\aA[E]{\x}:E\in\collectionGSF\}$\;
      $F_i\colon F_i\in\collectionGSF^\prime\text{ and }\mu(F_i)=\min\{\mu(F):F\in\collectionGSF^\prime\}$\;
      $\collectionGSF:=\collectionGSF\setminus\{E_i\}$\;
      $\collectionGSF^\prime:=\collectionGSF^\prime\setminus\{F_i\}$\;
      $i:=i+1$\;
	}\Return{$E_0,\dots,E_{\kappa-1}$, $F_0,\dots,F_{\kappa-1}$}\;
	}
\caption{Determination of sets $E_i$ and $F_i$, $i\in\ozn{\kappa-1}$
}
\label{alg_EF}
\end{algorithm}
\end{center}



\medskip
\noindent
\begin{minipage}{0.49\textwidth}
\SetKwRepeat{Do}{do}{while}
\begin{algorithm}[H]
\Fn{the-$(\cdot)$-map$(\collectionGSF, \mu, \x, \cA)$}{
$E_0,\dots,E_{\kappa-1}$, $F_0,\dots,F_{\kappa-1}$ $\leftarrow$ find-sets-$E_i$-and-$F_i$$(\collectionGSF, \mu, \x, \cA)$\;
\For{$(i=0$, $i<\kappa$, $i\mathrm{++})$}{%
    \For{$(j=0$, $j<\kappa$, $j\mathrm{++})$}{%
        \If{$(E_i=F_j^c)$}{%
            $(i):=j$\;
        }
    }
}
\Return{$(0),\dots,(\kappa-1)$}\;
}
\caption{Determination of values of a map $(\cdot)$
}
\label{alg_(i)}
\end{algorithm}
\end{minipage}
\hfill
\begin{minipage}{0.49\textwidth}
\SetKwRepeat{Do}{do}{while}
\begin{algorithm}[H]
\Fn{the-$\inv{\cdot}$-map$(\collectionGSF, \mu, \x, \cA)$}{
$E_0,\dots,E_{\kappa-1}$, $F_0,\dots,F_{\kappa-1}$ $\leftarrow$ find-sets-$E_i$-and-$F_i$$(\collectionGSF, \mu, \x, \cA)$\;
\For{$(i=0$, $i<\kappa$, $i\mathrm{++})$}{%
    \For{$(j=0$, $j<\kappa$, $j\mathrm{++})$}{%
        \If{$(E_i=F_j^c)$}{%
            $\inv{j}:=i$\;
        }
    }
}
\Return{$\inv{0},\dots,\inv{\kappa-1}$}\;
}
\caption{Determination of values of a map $\inv{\cdot}$
}
\label{alg_(j)^-1}
\end{algorithm}
\end{minipage}
\bigskip\bigskip

\smallskip

\noindent\textbf{Proof of Proposition~\ref{vlastnosti_psi}}
\begin{enumerate}[(i)]
\item 
Let us denote $$i^{\star}=\min\{l \in\ozn{\kappa-1}: \mu_{(l)}=\min_{k\leq i}\mu_{(k)}\}.$$
It is clear that $i^{\star}\leq i$, $\min_{k\leq i}\mu_{(k)}=\mu_{(i^{\star})}$. Further, 
\begin{itemize}
    \item Since $\mu_{(i^{\star})}\geq\min_{k\leq i^{\star}}\mu_{(k)}\geq \min_{k\leq i}\mu_{(k)}=\mu_{(i^{\star})}$, then $\min_{k\leq i^{\star}}\mu_{(k)}=\min_{k\leq i}\mu_{(k)}$. Therefore $\psil(i)\leq i^{\star}$.
    \item For any $l<i^{\star}$ it holds $$\min_{k\leq l}\mu_{(k)}>\min_{k\leq i}\mu_{(k)}.$$ Indeed, $\min_{k\leq l}\mu_{(k)}\geq\min_{k\leq i^{\star}}\mu_{(k)}=\mu_{(i^{\star})}$. However, $\min_{k\leq l}\mu_{(k)}\neq\min_{k\leq i^{\star}}\mu_{(k)}$. By contradiction: Let  $\min_{k\leq l}\mu_{(k)}=\mu_{(i^{\star})}$. Then there exist $k^{\star}\leq l<i^{\star}$ such that $\mu_{(k^{\star})}=\mu_{(i^{\star})}$. This is a contradiction with the definition of $j^{\star}$. Therefore $\psil(i)>l$.
\end{itemize}
From the previous it follows that $\psil(i)=i^{\star}$.

\item It follows directly from definitions of $\psir$, $\psil$. 

\item It follows from  the equalities $\min_{k\leq \psil(i)}\mu_{(k)}=\min_{k\leq i}\mu_{(k)}=\min_{k\leq \psir(i)}\mu_{(k)}$. Moreover, from part (i) and because of Lemma~\ref{vl_i} we have $\mu(\psil(i))=\min_{k\leq i}\mu_{(k)}=\bi{(i)}$.
\item It follows directly from definitions of $\psil$, $\psir$ and from the fact that if $\psil(a)=\psil(b)$  ($\psir(a)=\psir(b)$), then there is $\tilde{l}\in\ozn{\kappa-1}$ such that $\min_{k\leq a}\mu_{(k)}=\min_{k\leq\tilde{l}}\mu_{(k)}=\min_{k\leq b}\mu_{(k)}$.
\item It follows from the proof of (iv) and from the fact that if $\psil(a)<\psil(b)$, then $\min_{k\leq\psil(a)}\mu_{(k)}>\min_{k\leq\psil(b)}\mu_{(k)}$ (equality does not occur because of statement (v)).
Then we get $$\min_{k\leq a}\mu_{(k)}=\min_{k\leq\psil(a)}\mu_{(k)}>\min_{k\leq\psil(b)}\mu_{(k)}=\min_{k\leq b}\mu_{(k)},$$
where the first and the last equality hold because of definition $\psil$.
The same for $\psir(a)<\psir(b)$.
\item Let us denote $$M_1=\{l_1\in\ozn{\kappa-1}: \min_{k\leq l_1}\mu_{(k)}=\min_{k\leq a}\mu_{(k)}\}\,\quad M_2=\{l_2\in\ozn{\kappa-1}: \min_{k\leq l_2}\mu_{(k)}=\min_{k\leq b}\mu_{(k)}\}.$$
Since $a\leq b$, then $\min_{k\leq a}\mu_{(k)}\geq\min_{k\leq b}\mu_{(k)}$.
If $\min_{k\leq a}\mu_{(k)}=\min_{k\leq b}\mu_{(k)}$, then directly from definition $\psil$, resp.\ $\psir$, we have $\psil(a)=\psil(b)$, resp.\ $\psir(a)=\psir(b)$.
Further, let us suppose that $\min_{k\leq a}\mu_{(k)}>\min_{k\leq b}\mu_{(k)}$.
Then for any $l_1\in M_1$ and any $l_2\in M_2$ it holds
$$\min_{k\leq l_1}\mu_{(k)}=\min_{k\leq a}\mu_{(k)}>\min_{k\leq b}\mu_{(k)}=\min_{k\leq l_2}\mu_{(k)}\text{,}$$
therefore $l_2>l_1$.
Thus also $\min M_2>\min M_1$, resp.\ $\max M_2>\max M_1$, that is $\psil(b)>\psil(a)$, resp.\ $\psir(b)>\psir(a)$.
\qed
\end{enumerate}
\bigskip

\noindent\textbf{Proof of Proposition~\ref{naj_gsf_cez_psi}}
Because of Proposition~\ref{vlastnosti_psi} (iii) and Theorem~\ref{zjednodusenie_def} we know that the value $\min_{k\leq i}\mu_{(k)}=\mu_{(\psil(i))}$ is achieved on intervals $[\sA_m,\sA_{m+1})$, $m\in\{\psil(i),\dots,\psir(i)\}$. From definitions of $\psir$ and $\psil$ it is clear that there are no other intervals with this property. Therefore the value $\min_{k\leq i}\mu_{(k)}{=\mu_{(\psil(i))}}$ is achieved on $\bigcup_{m\in M}[\sA_m,\sA_{m+1})=\big[\sA_{\psil(i)},\sA_{\psir(i)+1}\big)$.
%
\qed
\bigskip

\noindent\textbf{Proof of Lemma~\ref{lema_ku_psi}}
Let us prove part (i). 
\begin{enumerate}
    \item[(i1)] If $\min_{k\leq i}\mu_{(k)}>\min_{k\leq i+1}\mu_{(k)}$, then directly from definitions of $\psil$ and $\psir$ we have $\psir(i)=i$ and $\psil(i+1)=i+1$, thus $\psir(i)+1=\psil(i+1)$.
    Further, according to Proposition~\ref{naj_gsf_cez_psi}, we have
    $$\gsf{}{\x}{\alpha}=\min_{k\leq i}\mu_{(k)}=\mu_{(\psil{(i)})}$$
    for any $\alpha\in\big[\sA_{\psil(i)},\sA_{\psir(i)+1}\big)$ and this interval is the greatest possible. 
    Using the above equality
    we have $$\big[\sA_{\psil(i)},\sA_{\psir(i)+1}\big)=\big[\sA_{\psil(i)},\sA_{\psil(i+1)}\big).$$ 
    \item[(i2)] It holds trivially.
\end{enumerate}
Part (ii) can be proved analogously.
\qed
\bigskip

\noindent{\textbf{Proof of Proposition~\ref{skratenie_cez_psi}}}
Let us consider an arbitrary (fixed) $i\in\ozn{\kappa-2}$. 
If $\min_{k\leq i}\mu_{(k)}>\min_{k\leq i+1}\mu_{(k)}$, then according to Lemma~\ref{lema_ku_psi} (i1) $\min_{k\leq i}\mu_{(k)}{=\mu_{(\psil(i))}}$ is achieved on $\big[\sA_{\psil(i)}, \sA_{\psil(i+1)}\big)$. 
Moreover, this interval is the greatest possible.  
If $\min_{k\leq i}\mu_{(k)}=\min_{k\leq i+1}\mu_{(k)}$, then $\psil(i)=\psil(i+1)$, see Lemma~\ref{lema_ku_psi}(i2), thus $\big[\sA_{\psil(i)}, \sA_{\psil(i+1)}\big)=\emptyset$. This demonstrates that each value of generalized survival function is included just once in the sum and the first formula is right. 
The third formula follows from Lemma~\ref{vl_i}. 
The second and the fourth formula can be proved analogously. 
\qed



\end{document}